\documentclass[11pt,a4paper]{amsart}
\usepackage{amsmath,amsfonts}
\usepackage{graphicx}
\usepackage{amssymb}
\usepackage{amsthm}
\usepackage{yfonts}
\usepackage{kpfonts}
\usepackage{tikz-cd}
\usepackage{stmaryrd}
\usepackage{amsaddr}

\usepackage{mathtools}
\mathtoolsset{showonlyrefs}

\numberwithin{equation}{section}

\usepackage[usenames,dvipsnames]{pstricks}
\usepackage{epsfig}
\usepackage{rotating}
\usepackage{pst-plot}
\usepackage{pst-eps}
\usepackage{pst-grad}
\usepackage{pstricks-add}
\usepackage{lmodern}
\usepackage{xcolor}
\usepackage{graphicx}
\usepackage[top=2cm,bottom=2cm,left=2cm,right=2cm,a4paper]{geometry}

\usepackage{etex}
\usepackage[all]{xy}

\usepackage{cite}

\usepackage{url}

\def\eps{{\varepsilon}}

\newcommand{\qq}{{\mathbb Q}}

\newcommand{\cc}{{\mathbb C}}

\newcommand{\pp}{{\mathbb P}}

\newcommand{\Pic}{{\rm{Pic}}}

\newcommand{\mbar}{\overline{\mathcal{M}}}

\theoremstyle{plain}
\newtheorem{theorem}{Theorem}[section]
\newtheorem{lemma}[theorem]{Lemma}
\newtheorem{proposition}[theorem]{Proposition}

\theoremstyle{definition}

\title{Pencils on surfaces with normal crossings and the Kodaira dimension of $\overline{\mathcal{M}}_{g,n}$}
\author{Daniele Agostini and Ignacio Barros}
\address[D. Agostini]{
	MPI for Mathematics in the Sciences\\
	Inselstrasse 22, 04103 - Leipzig\\
	Germany.} 
\email{daniele.agostini@mis.mpg.de}

\address[I. Barros]{
Laboratoire de math\'ematiques d'Orsay\\
Universit\'e Paris-Saclay\\
Rue Michel Magat, B\^at. 307\\
Orsay, 91405, France} 
\email{ignacio.barros-reyes@universite-paris-saclay.fr}

\begin{document}

\maketitle

\begin{abstract}
We study smoothing of pencils of curves on surfaces with normal crossings. As a consequence we show that the canonical divisor of $\overline{\mathcal{M}}_{g,n}$ is not pseudo-effective in some range, implying that $\overline{\mathcal{M}}_{12,6},\overline{\mathcal{M}}_{12,7},\overline{\mathcal{M}}_{13,4}$ and $\overline{\mathcal{M}}_{14,3}$ are uniruled. We provide upper bounds for the Kodaira dimension of $\overline{\mathcal{M}}_{12,8}$ and $\overline{\mathcal{M}}_{16}$. We also show that the moduli of $(4g+5)$-pointed hyperelliptic curves $\overline{\mathcal{H}}_{g,4g+5}$ is uniruled. Together with a recent result of Schwarz, this concludes the classification of moduli of pointed hyperelliptic curves with negative Kodaira dimension.\\

\end{abstract}


\nopagebreak



\section{Introduction}

It has long been established that for $g\geq2$, the moduli spaces $\overline{\mathcal{M}}_{g,n}$ are all of general type except for finitely many pairs $(g,n)$ occuring in relatively low genus. The description of the Kodaira dimension of $\overline{\mathcal{M}}_{g,n}$ 
is far from being complete, but it is unknown only for relatively small $g$ and $n$. Since Severi's conjecture \cite{Se} was disproven by Harris and Mumford \cite{HM}, computing the Kodaira dimension has become a major task in the study of moduli spaces of curves. See \cite{Se,CR1,CR2,EH1,EH2,BV,F, FJP,T, FV2} for an account on the results for $n=0$, and \cite{Lo, FPo, FV1, Be, BM,KT} for $n\geq1$. In the range $12\leq g\leq16$ there is no known example of moduli space $\overline{\mathcal{M}}_{g,n}$ of intermediate type. The state of the art in this range is summarized in the following table:
\begin{table}[h]
	\begin{tabular}{|l|l|l|l|l|l|}
		\hline
		\rule{0pt}{2.5ex}   &$\overline{\mathcal{M}}_{12,n}$&$\overline{\mathcal{M}}_{13,n}$&$\overline{\mathcal{M}}_{14,n}$&$\overline{\mathcal{M}}_{15,n}$&$\overline{\mathcal{M}}_{16,n}$ \\ \hline
		Uniruled&$n\leq 5$&$n\leq3$&$n\leq 2$&$n\leq 2$&\hbox{unknown}   \\ \hline
		Gen. Type&$n\geq 11$&$n\geq 11$&$n\geq 10$&$n\geq 10$&$n\geq 9$  \\ \hline
	\end{tabular}
\end{table}

Some of the moduli spaces in the above range are rationally connected or unirational, the latest contributions being \cite{V, BV} and recently \cite{KT}. Furthermore, very recently Farkas and Verra \cite{FV2} building on \cite{BV} showed that $\mbar_{16}$ is not of general type, i.e. the Kodaira dimension is bounded by $\operatorname{dim} \mbar_{16}-1$. \\

The standard argument in the literature to show that $\mbar_{g,n}$ is uniruled is to start with a general $n$-pointed curve and construct a surface $S$ such that $C$ moves on a pencil with the $n$ marked points in the base locus. This becomes significantly harder as either $g$ or $n$ grows. As an alternative one can show that $K_{\overline{\mathcal{M}}_{g,n}}$ is not pseudo-effective to conclude uniruledness, cf. \cite{BDPP}. In the range $5\leq g\leq 10$ this was used by Farkas and Verra \cite{FV1} to obtain uniruledness of $\overline{\mathcal{M}}_{g,n}$ for the highest $n$ known. The key point in this strategy is the computation of intersection numbers on $\mbar_{g,n}$ with respect to curves in $\mbar_{g,n}$ arising as a pencil on a smooth surface $S$. To do so, it is essential to have a complete understanding of the pencil $\Gamma$: our main technical result, Proposition \ref{prop:smoothing}, is a criterion for smoothing pencils on a reducible surface $S_1\cup S_2$. The analysis of the pencil can then be reduced to the two surfaces $S_i$, where it is easier.\\


As a consequence, we have our main contribution to the Kodaira classification of $\mbar_{g,n}$: 

\begin{theorem}
	\label{thm1}
	The moduli spaces $\overline{\mathcal{M}}_{12,7},\overline{\mathcal{M}}_{12,6},\overline{\mathcal{M}}_{13,4}$ and $\overline{\mathcal{M}}_{14,3}$ are uniruled.
\end{theorem}

Moreover, we can provide the following bound for $\mbar_{12,8}$.

\begin{theorem}
	\label{thm3}
	The Kodaira dimension of $\overline{\mathcal{M}}_{12,8}$ is bounded by $\dim \overline{\mathcal{M}}_{12,8}-2$.
\end{theorem}  

A mistake in \cite[Proof of Thm. 0.1]{CR3} asserting the uniruledness of $\overline{\mathcal{M}}_{16}$ was recently found \cite{T}, leaving open the question about the Kodaira dimension of $\mbar_{16}$. Farkas and Verra \cite{FV2} recently proved that $\mbar_{16}$ is not of general type. We reprove their result with our methods and we improve the bound on the Kodaira dimension by one.

\begin{theorem}
	\label{thm2}
	The Kodaira dimension of $\overline{\mathcal{M}}_{16}$ is bounded by $\dim\overline{\mathcal{M}}_{16}-2$. 
\end{theorem} 

Finally, in \cite{Sc} it is proved that the moduli space of pointed hyperelliptic curves $\overline{\mathcal{H}}_{g,n}$ of genus $g\geq 2$ is of general type for $n\geq 4g+7$ and it has non-negative Kodaira dimension for $n=4g+6$. Moreover, it was known to be rational for $n\leq 2g+8$ and uniruled for $n\leq 4g+4$, cf. \cite{Ca, Be}. Our result is that $\overline{\mathcal{H}}_{g,4g+5}$ is uniruled.

\begin{proposition}
	\label{prop1}
	The moduli space of pointed hyperelliptic curves $\mathcal{H}_{g,4g+5}$ is covered by rational surfaces for $g\geq 2$.
\end{proposition}

The basic principle in our proofs is the following: when $\overline{\mathcal{M}}_{g,n}$ is covered by rational curves $\Gamma\subset \overline{\mathcal{M}}_{g,n}$, the family of curves $\mathcal{X}\longrightarrow \Gamma$ is a covering surface for $\overline{\mathcal{M}}_{g,n+1}$ and the same holds for $\mathcal{H}_{g,n}$. Proposition \ref{prop1} is obtained by simply observing that the family $\mathcal{X}$ over the covering rational curve of $\mathcal{H}_{g,4g+4}$ constructed in \cite{Be} is a rational surface.\\ 

For genera $13$ and $14$, we apply the same principle to the constructions of \cite{Be} and \cite{V,BV} but in these cases the covering surface $\mathcal{X}$ are no longer rational, they are in fact of general type. However, we can find positive curves on (covers of) $\mathcal{X}$ that intersect negatively the canonical divisor
of $\overline{\mathcal{M}}_{g,n+1}$ to obtain uniruledness. This comes at the cost of having to compute intersection numbers in $\overline{\mathcal{M}}_{g,n}$ for which, as we said before, it is essential to have a complete understanding of the singular elements in the rational curve $\Gamma\subseteq \mbar_{g,n}$. 

In particular we have to rule out unstability and to do so we need to analyze carefully the curve $\Gamma$. In the cases of genera $13$ and $14$ a general curve $[C]\in\overline{\mathcal{M}}_{g}$ sits in a smooth canonical surface $S\subseteq \pp^n$ in such a way that the linear system $|\mathcal{O}_S(C)|$ on $S$ is positive dimensional. Then $\Gamma$ is induced by a pencil in this linear system and to study its properties we specialize to particular surfaces $S$. Specializing to particular surfaces to obtain information about behavior of general curves has been extensively used throughout the literature on curves and their moduli, with rational surfaces and K3 surfaces having a distinguished role. Inspired by \cite{V, BV} we specialize to canonical normal crossing surfaces 
$$S_1\cup S_2\subset \pp^r$$
where $S_1$ and $S_2$ are rational. \\

In Section \ref{sec2} we develop the aforementioned criteria for smoothing pencils on any normal crossing surface $S_1\cup S_2$ and apply these to obtain Theorem \ref{thm1}, for genus $13$ and $14$. We hope that these criteria could be useful for further applications. \\

To deduce the remaining results in  genus $12$ we exploit pencils on $K3$ surfaces. The strategy in this case is inspired by \cite{CR3}, where they obtain nef classes $\mbar_{g,n}$ by taking a family of curves on $\mbar_{g-1,n+2}$ and glue together the two last marked points.\\

Finally, the result in genus $16$ is obtained in the same way, looking at the marked curves in $\mbar_{15,2}$ constructed in \cite{BV,FV2}. \\

The structure of the paper is as follows. In Section \ref{sec1} we state the general principle more precisely and we prove Proposition \ref{prop1}. Moreover, we construct nef curve classes $\Theta$ on $\overline{\mathcal{M}}_{g,n+k}$ coming from a family of genus $g$ curves $\mathcal{X}$ over a pencil $\Gamma$ in $\overline{\mathcal{M}}_{g,n}$. We also compute the intersections of $\Theta$ with the generators of the Picard group of $\overline{\mathcal{M}}_{g,n+k}$ assuming that $\Gamma$ is a good pencil in the sense of condition $(\star\star)$ below. In Section \ref{sec2} we study in general pencils on normal crossing surfaces and establish criteria for their smoothability. In Section \ref{sec3} we apply the previous results for $g=13,14$ and deduce Theorem \ref{thm1} in these cases. In Section \ref{sec5} we use pencils on $K3$ surfaces to deduce the remaining results in $g=12$ and finally in Section \ref{sec4} we combine the various approaches to prove Theorem \ref{thm2} for genus $16$.

\section{A general principle}
\label{sec1}

Let $\mathcal{M}$ be a moduli space of curves and $\mathcal{C}$ the corresponding universal family. We will keep this vague but having in mind that $\mathcal{M}$ is the moduli space of pointed curves $\mathcal{M}_{g,n}$ or pointed hyperelliptic curves $\mathcal{H}_{g,n}$. We start by discussing a very elementary principle that allows us to lift curves covering $\mathcal{M}$ to varieties covering $\mathcal{C}$. Suppose that $\Gamma$ is a curve with a nonconstant map
\[
\begin{tikzcd}  
\Gamma \arrow[r,dashed] &\mathcal{M}.
\end{tikzcd}  
\]
We can complete this diagram naturally to a fibered square
\begin{equation}
\begin{tikzcd}
\mathcal{X}\arrow[d,dashed]\arrow[r,dashed]&\mathcal{C}\arrow[d]\\
\Gamma\arrow[r,dashed] &\mathcal{M}.
\end{tikzcd}
\end{equation}
in other words, $\mathcal{X}$ is the family induced by the curve $\Gamma$.  The obvious observation is that
\begin{itemize}
\item[$\left(\star\right)$]
if the curve $\Gamma$ passes through a general point of $\mathcal{M}$, then the surface $\mathcal{X}$ passes through a general point on $\mathcal{C}$. \label{p1}
\end{itemize}

This means that there exists a fibration $\mathcal{Y}\to B$ of generically relative dimension $\dim \mathcal{X}$, such that $\mathcal{X}$ sits in $\mathcal{Y}$ as the central fiber and $\mathcal{Y}$ dominates $\mathcal{C}$. A situation where the principle works particularly well is when $\Gamma$ is a pencil of genus $g$ curves on a smooth surface $S$. 
In this case, the variety $\mathcal{X}$ is the incidence correspondence
\begin{equation}
\mathcal{X} = \{ (C,p) \in S \times \Gamma \,\,|\,\, p \in C \}
\end{equation} 
and we see by construction that the projection map $\mathcal{X} \to S$ is birational, since it is an isomorphism outside of the base points of $\Gamma$. Hence, we see that in this case the universal family $\mathcal{C}$ is covered by surfaces birational to $S$.




As an application, we can immediately prove:

\begin{proposition}
	The moduli space of pointed hyperelliptic curves $\mathcal{H}_{g,4g+5}$ for $g\geq 2$ is covered by rational surfaces.
\end{proposition}

\begin{proof}
	We look at $\mathcal{H}_{g,4g+5}$ as the universal family over the moduli space $\mathcal{H}_{g,4g+4}$. Benzo shows in \cite[Prop. 4.1]{Be} that the latter is uniruled as follows: if $(C,p_1,\dots,p_{4g+4})$ is a general pointed hyperelliptic curve, then $C$ can be realized as a divisor on $\mathbb{P}^1\times \mathbb{P}^1$ in the linear system of $L=\mathcal{O}(2,g+1)$. Furthermore, the linear system $\Gamma=|L\otimes \mathcal{I}_{p_1,\dots,p_{4g+4}}|$ of curves in $|L|$ passing through the points $p_1,\ldots, p_{4g+4}$ defines a pencil on $\mathbb{P}^1\times \mathbb{P}^1$ and $C$ belongs to it by construction. Then the previous principle shows that $\mathcal{H}_{g,4g+5}$ is covered by surfaces birational to $\mathbb{P}^1\times \mathbb{P}^1$, and thus rational. 
\end{proof}

In the rest of the paper, we will take as $\mathcal{M}$ the moduli space $\overline{\mathcal{M}}_{g,n}$ of stable $n$-marked curves of genus $g$. The universal family over it is then naturally identified with $\overline{\mathcal{M}}_{g,n+1}$, via the map $\pi \colon \overline{\mathcal{M}}_{g,n+1} \longrightarrow \overline{\mathcal{M}}_{g,n}$ that forgets the last marked point. According to the general principle we want to lift curves covering $\overline{\mathcal{M}}_{g,n}$ to surfaces covering $\overline{\mathcal{M}}_{g,n+1}$. As before, this works particularly well in the case of \emph{good pencils}, defined as follows:
 
\begin{itemize} 
	\item [$(\star\star)$] A general pointed curve $(C,p_1,\dots,p_n) \in \overline{\mathcal{M}}_{g,n}$ lies in a smooth surface $S$ such that the linear system $\Gamma = |\mathcal{O}_S(C)\otimes \mathcal{I}_{p_1,\dots,p_n}|$ is a non-isotrivial pencil of genus $g$ curves. We assume moreover that the base locus of $\Gamma$ is reduced, meaning that two general curves in $\Gamma$ intersect transversally, and that for every curve $C'$ in $\Gamma$, the $n$-pointed curve $(C',p_1,\ldots,p_n)$ is stable. We also add the last technical condition that the induced curve $\Gamma\to\overline{\mathcal{M}}_{g,n}$ does not intersect the boundary divisor $\Delta_{1:\varnothing}$. 
\end{itemize}  

The last assumption is not strictly necessary but makes formulas simpler, see Lemma \ref{lemma1.sec1} and Proposition \ref{prop2.sec1}. See after the proof of Proposition \ref{prop3.sec1} for a reminder on the definition of boundary divisors in $\overline{\mathcal{M}}_{g,n}$.\\ 

Now we place ourselves in the situation $(\star\star)$. The incidence correspondence 
$$\mathcal{X} = \{ (C,p) \in \Gamma \times S \,|\, p\in C \}$$ 
is identified with the blow-up of $S$ at the base points of $\Gamma$, via the projection map $\mathcal{X} \longrightarrow S$, and thanks to our hypotheses in $(\star\star)$ we have a cartesian diagram
\begin{equation}\label{eq:diagrammgn}
\begin{tikzcd}
\mathcal{X}\arrow[d]\arrow[r,"f"]&\overline{\mathcal{M}}_{g,n+1}\arrow[d,"\pi"]\\
\Gamma\arrow[r] &\overline{\mathcal{M}}_{g,n}.
\end{tikzcd}
\end{equation}
The surface $\mathcal{X}$ passes through a general point of $\overline{\mathcal{M}}_{g,n+1}$ and it is birational to $S$. However, in many cases, the surface $S$ is of general type so that this does not give us information on the Kodaira dimension of $\overline{\mathcal{M}}_{g,n+1}$. However, this construction provides us with a wealth of curves covering $\overline{\mathcal{M}}_{g,n+1}$, namely all those covering $S$. In particular, any smooth curve $D \subseteq S$ which avoids the base points of $\Gamma$ can be lifted isomorphically to a curve $D'$ in $\mathcal{X}$. Furthermore, if $D$ is a covering curve in $S$ and  
\[ (D'\cdot f^*K_{\overline{\mathcal{M}}_{g,n+1}}) < 0, \] 
then $\overline{\mathcal{M}}_{g,n+1}$ is covered by curves that intersect the canonical class negatively. Hence the canonical divisor class of $\overline{\mathcal{M}}_{g,n+1}$ is not pseudoffective, so that $\overline{\mathcal{M}}_{g,n+1}$ must be uniruled thanks to \cite{BDPP}. We can also iterate this strategy: assume we are in the situation $(\star\star)$ and let $D_1,\ldots,D_k$ be a collection of smooth curves in $S$ away from the base locus of the pencil $\Gamma$, intersecting pairwise transversally and such that $(D_i\cdot C)\geq 1$ for a general element $C\in \Gamma$. Then the curves $D_i$ lift isomorphically to $\mathcal{X}$, so that we have natural maps $D_i\to\Gamma$.  Consider the base change diagram
\begin{equation}
\label{eq5.sec1}
\begin{tikzcd}
\mathcal{Y}\arrow[r]\arrow[d]&\mathcal{X}\arrow[r]\arrow[d]&S\\
\Theta:=D_1\times_{\Gamma}\cdots\times_{\Gamma}D_k\arrow[r]&\Gamma.&
\end{tikzcd}
\end{equation} 
The map $\mathcal{Y}\longrightarrow \Theta$ comes with $n+k$ many sections $s_1,\dots,s_{n+k}$, the first $n$ pulled back from $\Gamma\longrightarrow \mathcal{X}$ and the last $k$ induced by diagonals 
$$s_{n+i}\colon (p_1,\ldots,p_k)\mapsto (p_1,\ldots,p_k, p_i)\in \mathcal{Y},$$
for $i=1,\ldots,k$. After blowing up $\widetilde{\mathcal{Y}}\longrightarrow\mathcal{Y}$ at all the pairwise intersection points of the last $k$ sections, and assuming the proper transform of $D_i$ does not meet the singular points of the fibers, we obtain a family of stable pointed curves
\begin{equation}
\label{eq6.sec1}
\begin{tikzcd}
\widetilde{\mathcal{Y}}\arrow[r]\arrow[d]&\overline{\mathcal{M}}_{g,n+k+1}\arrow[d]\\
\Theta\arrow[r]\arrow[d]&\overline{\mathcal{M}}_{g,n+k}\arrow[d]\\
\Gamma\arrow[r]&\overline{\mathcal{M}}_{g,n}.
\end{tikzcd}
\end{equation}

Furthermore, if the curves $D_1,\dots,D_k$ cover $S$, then the curve $\Theta$ covers $\mbar_{g,n+k}$:
\begin{proposition}
	\label{prop3.sec1}
	Under the assumptions $(\star\star)$, let $D_1,\dots,D_k \subseteq S$ be smooth covering curves that are away from the base points of $\Gamma$, that intersect pairwise transversally and whose proper transforms in $\mathcal{X}$ do not meet the singular locus of the fibers of $\mathcal{X}\to \Gamma$. Then the numerical class of the obtained curve $\Theta\in N_1(\overline{\mathcal{M}}_{g,n+k})$ is nef.
\end{proposition}
Recall the following fundamental fact: Let $X$ be a smooth projective variety and $\Xi\in N_1(X)$ a fixed curve class. Assume for a general point $p\in X$ there is an irreducible curve $C\subset X$ passing through $p$ with fixed numerical class $C\equiv \Xi$, then $\Xi$ is {\textit{nef}}, i.e., $(\Xi\cdot D)\geq0$ for all effective divisors $D\in N^1(X)$. Indeed, for any effective divisor $D$ in $X$, if we fix a general point $p\in X$ outside $D$, then since $C$ passes through $p$, $C\cdot D\geq0$. The same holds for coarse moduli spaces of smooth stacks where divisors are taken with $\qq$-coefficients.  

\begin{proof}[Proof of Proposition \ref{prop3.sec1}]
	We iterate the principle established before. By assumption $\Gamma\subset \overline{\mathcal{M}}_{g,n}$ is a covering curve and therefore the universal surface $\mathcal{X}$ is a covering surface for $\overline{\mathcal{M}}_{g,n+1}$. Since $D_1$ covers $\mathcal{X}$, the curve $D_1\to\overline{\mathcal{M}}_{g,n+1}$ is a covering curve. Again we look at the universal surface $\mathcal{X}_1$ over $D_1$, this is a finite base change $\mathcal{X}_1\to \mathcal{X}$ induced by the finite map $D_1\to \Gamma$. Since $D_2$ is also a covering curve for $\mathcal{X}$, its pullback to $\mathcal{X}_1$ is a covering curve for $\mathcal{X}_1$ and therefore for $\overline{\mathcal{M}}_{g,n+1}$. Iterating this process, since all the choices made do not change the numerical class of the resulting curve (see Proposition \ref{prop2.sec1}) the resulting curve $\Theta=D_1\times_\Gamma \ldots\times_\Gamma D_k$ is a covering curve for $\overline{\mathcal{M}}_{g,n+k}$. 
\end{proof}

In particular, if $\Theta$ intersects the canonical class of $\mbar_{g,n+k}$ negatively, then this is not pseudoeffective so that $\mbar_{g,n+k}$ is uniruled.


To pursue this strategy, we need to compute some intersection numbers. Recall \cite{ACG} that when $g\geq 3$, the group $\Pic_{\mathbb{Q}}(\overline{\mathcal{M}}_{g,n})$ is freely generated by the class $\lambda$, the $\psi$-classes $\psi_i$ for $i=1,\dots,n$ 
and the classes of the irreducible components of the boundary. One component $\Delta_{\rm{irr}}$ corresponds to irreducible nodal curves. Instead, for any $0\leq i \leq g$ and any subset $S\subset\{1,\dots,n\}$ we denote by $\Delta_{i:S}$ the component whose general point is a $1$-nodal reducible curve with two components such that one component has genus $i$ and contains precisely the markings of $S$. We observe that $\Delta_{i:S}=\Delta_{g-i:S^c}$ and we require $|S|\geq 2$ for $i=0$ and $|S|\leq n-2$ for $i=g$. One then defines the classes $\delta_{\rm{irr}} = [\Delta_{\rm{irr}}],\delta_{i:S}=[\Delta_{i:S}]$ and $\delta = \delta_{\rm{irr}}+\sum_{i,S} \delta_{i:S}$

The canonical class of $\overline{\mathcal{M}}_{g,n}$ can be computed using Grothendieck-Riemann-Roch and it is given by  the formula 
\begin{equation}
\label{canMgn}
K_{\overline{\mathcal{M}}_{g,n}}=13\cdot\lambda+\sum_{i=1}^n\psi_i -2\cdot \delta -\delta_{1:\varnothing}.
\end{equation} 
Finally we recall \cite{ACG, CR3} how to compute the intersection of these classes with a smooth test curve $\Gamma \longrightarrow \overline{\mathcal{M}}_{g,n}$. This curve corresponds to a smooth family $\mathcal{X}\longrightarrow \Gamma$ of stable genus $g$ curves together with $n$ disjoint sections $s_i:\Gamma\longrightarrow\mathcal{X}$. The intersection numbers are

\begin{align}
\label{standard}
(\Gamma\cdot\lambda)&=\chi(\mathcal{O}_\mathcal{X})-\chi(\mathcal{O}_\Gamma)\cdot(1-g), \\ (\Gamma\cdot\delta)&=\chi_{\rm{top}}(\mathcal{X})-\chi_{\rm{top}}(\Gamma)\cdot(2-2g), \\  
(\Gamma\cdot \psi_i) &= -(s_i(\Gamma)^2). 
\end{align}
Here $\chi_{\rm{top}}$ stands for the topological Euler characteristic.

Now we can compute the intersection numbers we are interested in. First we start with the situation of \eqref{eq:diagrammgn}:

\begin{lemma}\label{lemma1.sec1}
In the assumptions of $(\star\star)$, let $D\subseteq S$ be a smooth and irreducible curve which avoids the base locus of $\Gamma$ as well as the nodes of the curves in $\Gamma$. Let us denote by $D'$ its proper transform in $\mathcal{X}$ and $C$ a general fiber of $\mathcal{X}\to \Gamma$. Then
\begin{enumerate}
	\item[(i)] $(D'\cdot f^*\pi^*\alpha ) = (D\cdot C)(\Gamma \cdot \alpha)$ for every $\alpha \in \operatorname{Pic}_{\mathbb{Q}}(\overline{\mathcal{M}}_{g,n})$.
	\item[(ii)] $(D'\cdot f^*\psi_{n+1}) = (D\cdot K_S)+2(D\cdot C)$.
	\item[(iii)] $D'$ intersects trivially all the boundary divisors $\delta_{0:\{i,n+1\}}$ of $\overline{\mathcal{M}}_{g,n+1}$.
\end{enumerate}

In particular
\[ (D'\cdot f^*K_{\overline{\mathcal{M}}_{g,n+1}}) = (D\cdot C) \left(\Gamma\cdot K_{\overline{\mathcal{M}}_{g,n}}+2+\frac{(K_S\cdot D)}{(D\cdot C)}\right) \]
and if $D$ is a covering curve and this is negative, then $\mbar_{g,n+1}$ is uniruled.
\end{lemma}

\begin{proof}
We observe that the family $\mathcal{X}_{D'} \longrightarrow D'$ consists of at worst nodal curves, since the same holds for the family $\mathcal{X} \longrightarrow \Gamma$. Furthermore, thanks to our assumptions on $D$, the sections $s_1(D'),\dots,s_{n+1}(D')$, defined after \eqref{eq5.sec1}, are pairwise disjoint and none of them passes through the nodes of the fibers of $\mathcal{X}_{D'} \longrightarrow D'$. Hence, the family $\mathcal{X}_{D'}$ over $D'$ is already a family of stable and irreducible $(n+1)$-pointed curves. Now we compute the various intersection numbers:

	\begin{enumerate}
		\item[(i)] This follows from the projection formula together with the fact that map $\mathcal{X}\longrightarrow \Gamma$ has degree $(D\cdot C)$ when restricted to the curve $D'$.
		\item[(ii)]  We have a cartesian diagram
		\begin{equation}
		\label{eq4.sec1}
		\begin{tikzcd}
		\mathcal{X}_{D'}\arrow[r] \arrow[d]&\mathcal{X}\arrow[d]\\
		D'\arrow[r, "h"]&\Gamma
		\end{tikzcd}
		\end{equation}
	    and since $D$ avoids the base locus of $\Gamma$, as well as the nodes of the curves in $\Gamma$ it follows that $\mathcal{X}_{D} \subseteq D'\times \mathcal{X}$ is a smooth divisor of class $\operatorname{pr}_1^*h^*\mathcal{O}_{\Gamma}(1) + \operatorname{pr}_2^*\mathcal{O}_{\mathcal{X}}(C')$ where $C'\subseteq \mathcal{X}$ is the proper transform of $C$. Then the adjunction formula shows that  
		\begin{align*} 
		D'\cdot f^*\psi_{n+1} & = -(s_{n+1}(D')\cdot s_{n+1}(D')) = \deg \left(s_{n+1}^* K_{\mathcal{X}_{D'}} - K_{D'}\right) \\
		& = \deg \left( s_{n+1}^*\operatorname{pr}_1^*(K_{D'}+h^*\mathcal{O}_{\Gamma}(1))-K_{D'}\right) \\
		&+ \deg \left(s_{n+1}^*\operatorname{pr}_2^*(\mathcal{O}_{\mathcal{X}}(C')+K_{\mathcal{X}})\right) \\
		& = (C\cdot D) + (D'\cdot (C'+K_{\mathcal{X'}})) = 2(C\cdot D)+(D\cdot K_S).
		\end{align*}
		In the third equality we have used that the canonical class on $\mathcal{X}_{D'}$ can be computed by adjunction as $K_{\chi_{D'}} = \operatorname{pr}_1^*(K_{D'}+h^*\mathcal{O}_{\Gamma}(1)) + \operatorname{pr}_2^*{K_{\chi}+\mathcal{O}_{\chi}(C')}$. In the fourth equality instead, we have used that the composition $\operatorname{pr}_1\circ s_{n+1}$ is the identity, whereas $\operatorname{pr}_2\circ s_{n+1}$ is the inclusion of $D'$ inside $\mathcal{X}_{D'}$.
		
		\item[(iii)] Since $D$ does not meet the base locus of $\Gamma$, the assertion follows.
	\end{enumerate}
 Now we compute the intersection with the canonical class. We know from \eqref{canMgn} and standard formulas for pull-backs of divisor classes via the map $\pi\colon \overline{\mathcal{M}}_{g,n+1}\longrightarrow\overline{\mathcal{M}}_{g,n}$, cf. \cite[Chapter 17]{ACG}, that 
 \[ K_{\overline{\mathcal{M}}_{g,n+1}}=\pi^* K_{\overline{\mathcal{M}}_{g,n}}+\psi_{n+1}-2\sum_{i=1}^n\delta_{0:\{i,n+1\}}+\delta_{1:\{n+1\}}. \]
Hence, the previous points together with assumption $(\star\star)$ show that
 \begin{align*}
(D'\cdot f^*K_{\overline{\mathcal{M}}_{g,n+1}}) = (D\cdot C)(\Gamma\cdot K_{\overline{\mathcal{M}}_{g,n+1}}) + 2(D\cdot C) + (K_S\cdot D)
 \end{align*} 
which is what we wanted to prove.

To conclude, the assumptions show that $f_*D'$ yields a nef curve class on $\overline{\mathcal{M}}_{g,n+1}$, so that if $(f_*D'\cdot K_{\overline{\mathcal{M}}_{g,n+1}})$ is negative then $\overline{\mathcal{M}}_{g,n+1}$ is uniruled by \cite{BDPP}.
\end{proof}

Then we consider the more general situation of \eqref{eq6.sec1}:

\begin{proposition}
\label{prop2.sec1}
Under the assumptions $(\star\star)$, let $D_1,\dots,D_k \subseteq S$ be smooth covering curves that are away from the base points of $\Gamma$, intersect pairwise transversally, and whose proper transforms in $\mathcal{X}$ do not meet the singular locus of the fibers. Then the following intersection products hold for $\Theta$ in $\overline{\mathcal{M}}_{g,n+k}$:
\begin{enumerate}
	\item[(i)] $(\Theta\cdot\alpha)=\left(\prod_{i=1}^k (D_i\cdot C)\right)\cdot \left(\Gamma\cdot\alpha\right)$
	for all $\alpha\in \pi^*\Pic_{\qq}\left(\overline{\mathcal{M}}_{g,n}\right)$, where $\pi$ is the map that forgets the last $k$ marked points.
	\item[(ii)] $(\Theta\cdot\psi_{n+i})= \prod_{j\neq i}(D_j\cdot C)\cdot\left[(K_S\cdot D_i)+2(D_i\cdot C)\right]+\sum_{1\leq a<b\leq k}(D_a\cdot D_b)$ for all $i=1,\dots,k$.
	\item[(iii)] $(\Theta\cdot\delta_{0:\{n+i,n+j\}})= (D_i\cdot D_j)$ for all $1\leq i < j \leq k$, and $\Theta$ intersects trivially all the boundary components $\delta_{0:\{i, n+j\}}$ for $i=1,\ldots,n$.
\end{enumerate} 
In particular 
\begin{align}
(\Theta\cdot K_{\overline{\mathcal{M}}_{g,n+k}})&=\prod_{i=1}^k (D_i\cdot C) \cdot \left((\Gamma\cdot K_{\overline{\mathcal{M}}_{g,n}})+2k+\sum_{i=1}^k\frac{(K_S\cdot D_i)}{(D_i\cdot C)}\right)\\
&-\sum_{1\leq i<j\leq k}(D_i\cdot D_j).
\end{align}
and if this is negative, then $\overline{\mathcal{M}}_{g,n+k}$ is uniruled.
\end{proposition}

\begin{proof}
Projection formula gives us (i). Let $\pi_{n+i}:\overline{\mathcal{M}}_{g,n+k}\to\overline{\mathcal{M}}_{g,n+1}$ be the map that remembers the first $n$ marked points and the $(n+i)$-th one. Consider the diagram
\begin{equation}
\label{eq8.sec1}
\begin{tikzcd}
\widetilde{\mathcal{Y}}\arrow[r, "Bl"]&\mathcal{Y}\arrow[r]\arrow[d]&\mathcal{X}_{D_i}\arrow[r]\arrow[d]&\mathcal{X}\arrow[d]\\
&\Theta\arrow[r]&D_i\arrow[r]&\Gamma.
\end{tikzcd}
\end{equation} 
The push forward $\pi_{n+i,*}(\Theta)$ in $\overline{\mathcal{M}}_{g,n+1}$ is given by 
$$\left(\prod_{j\neq i}(D_j\cdot C)\right) D_i,$$
where the curve $D_i\longrightarrow \overline{\mathcal{M}}_{g,n+1}$ is induced by the middle vertical arrow in \eqref{eq8.sec1}. From Lemma \ref{lemma1.sec1} it follows that if $\Delta_i\subset \mathcal{Y}$ is the diagonal section, then
$$-(\Delta_i^2)=\prod_{j\neq i}(D_j\cdot C)\cdot\left[(K_S\cdot D_j)+2(D_j\cdot C)\right].$$ 
After blowing up $\mathcal{Y}$ at the intersection of the diagonal sections $\Delta_1,\ldots,\Delta_k$, one obtains (ii) and (iii). For the canonical class, we see that 
if $\pi:\overline{\mathcal{M}}_{g,n+k}\to\overline{\mathcal{M}}_{g,n}$ if the map that forgets the last $k$ marked points, then
\begin{align*} 
K_{\overline{\mathcal{M}}_{g,n+k}}&=\pi^\star K_{\overline{\mathcal{M}}_{g,n}}+\sum_{i=1}^k\psi_{n+i}-2\sum_{1\leq i<j\leq k}\delta_{0:\{n+i,n+j\}} \\
&+\left(\sum_{S\subset\{n+1,\ldots,n+k\}}\delta_{1:S}-2\sum_{S}\delta_{0:S}\right).
\end{align*}
where $\delta_{0:S}$ denote the remaining $\delta_0$-boundary components intersecting $\Theta$ trivially. By assumption $(\star\star)$, $(\Theta\cdot\delta_{1:S})=0$ and the formula follows.

To conclude, the assumptions show that $\Theta$ yields a nef curve class on $\overline{\mathcal{M}}_{g,n+k}$, so that if $(\Theta\cdot K_{\overline{\mathcal{M}}_{g,n+k}})$ is negative then $\overline{\mathcal{M}}_{g,n+k}$ is uniruled by \cite{BDPP}.  
\end{proof}





\section{Smoothing pencils on reducible surfaces}\label{sec2}

Now we need to find good pencils satisfying condition $(\star\star)$. The trickiest part is to check that all curves in the linear system $\Gamma$ are stable. To do so, we will construct a pencil on a reducible surface $S_1\cup S_2$ with this property, and we will show that this deforms to a pencil on a smooth surface.

More precisely, let us consider two smooth and irreducible surfaces $S_1,S_2\subseteq \mathbb{P}^r$ which meet transversely along a smooth and irreducible curve $B$. We also consider two smooth and irreducible curves $C_1\subseteq S_1, C_2\subseteq S_2$ which intersect the curve $B$ transversely. We have the intersections
\[ Z_1 = (C_1\cap B)\setminus C_2, \qquad Z_2 = (C_2\cap B)\setminus C_1, \qquad W = C_1\cap C_2  \]  
that we can also consider as divisors on $B$. We observe that the divisor $C_o=C_1\cup C_2$ on the reducible surface $S_o = S_1\cup S_2$ fails to be Cartier exactly at the points of $Z_1,Z_2$: we can fix this by considering the blow-up $\varepsilon_o \colon \widetilde{S}_o \to S_o$ along these points. This can be seen also as the blow-ups $\varepsilon_i\colon \widetilde{S}_i \to S_i$ along $Z_i$ for $i=1,2$ glued together along the proper transform of $B$, that we denote with the same letter. We see that the proper transforms of $C_1$ and $C_2$, that we denote by $\widetilde{C}_1\subseteq \widetilde{S}_1,\widetilde{C}_2\subseteq \widetilde{S}_2$, intersect $B$ precisely at the points of $W$. Hence, $\widetilde{C}_o=\widetilde{C}_1\cup \widetilde{C}_2$ is a Cartier divisor on $\widetilde{S}_o$, and the corresponding line bundle fits into an exact sequence
\[ 0 \longrightarrow \mathcal{O}_{\widetilde{S}_o}(\widetilde{C}_o)  \longrightarrow \mathcal{O}_{\widetilde{S}_1}(\widetilde{C}_1)\oplus \mathcal{O}_{\widetilde{S}_2}(\widetilde{C}_2) \longrightarrow \mathcal{O}_{B}(W) \longrightarrow 0 \]  
where the last map is $(\widetilde{\sigma}_1,\widetilde{\sigma}_2)\mapsto \widetilde{\sigma}_{1|B}-\widetilde{\sigma}_{2|B}$. Pushing forward along $\varepsilon_o$ we obtain another exact sequence
\begin{equation}\label{eq:pushforwardF} 
0 \longrightarrow \varepsilon_{o,*}\mathcal{O}_{\widetilde{S}_o}(\widetilde{C}_o)  \longrightarrow (\mathcal{O}_{S_1}(C_1)\otimes \mathcal{I}_{Z_1,S_1})\oplus(\mathcal{O}_{S_2}(C_2)\otimes \mathcal{I}_{Z_2,S_2}) \longrightarrow \mathcal{O}_{B}(W) \longrightarrow 0 
\end{equation}  
where the map on the right is the difference of the canonical restriction maps $\mathcal{O}_{S_i}(C_i)\otimes \mathcal{I}_{Z_i} \longrightarrow \mathcal{O}_B(C_i-Z_i)\cong \mathcal{O}_B(W)$. In particular this shows that the sequence \eqref{eq:pushforwardF} is right exact. Then, each section in $H^0(S_o,\varepsilon_{o,*}\mathcal{O}_{\widetilde{S}_o}(\widetilde{C}_o))$ which is nonzero on both $S_1,S_2$ corresponds to an union of curves $C_i'\subseteq S_i$ such that $C_i'\cap B = Z_i\cup W'$ where $W'=C_1'\cap C_2'$ is a divisor in $|\mathcal{O}_B(W)|$.
\vskip.3in

Now the idea is that a pencil of reducible curves in $H^0(S_o,\varepsilon_*\mathcal{O}_{\widetilde{S}_o}(\widetilde{C}_o))$ can be smoothed out to a pencil on a smooth surface.

\begin{proposition}\label{prop:smoothing}
	In the above notation, let $\mathcal{C}\subseteq \mathcal{S}$ be two flat families over a small pointed disk $(\Delta,o)$  such that
	\begin{enumerate}
		\item[(a)] $\mathcal{S}_o \cong S_o$ and $\mathcal{C}_o \cong C_o$.
		\item[(b)] $\mathcal{S}_t,\mathcal{C}_t$ are smooth for $t\ne o$.
		\item[(c)] $\mathcal{C}$ is smooth.
	\end{enumerate}
    Let $\varepsilon\colon \widetilde{\mathcal{S}} \to \mathcal{S}$ be the blow up along $Z_1\cup Z_2$ and let $\widetilde{\mathcal{C}}$ be the proper transform of $\mathcal{C}$. Then $\widetilde{\mathcal{C}}$ is a Cartier divisor on $\widetilde{\mathcal{S}}$ and the restriction of the sheaf $\varepsilon_*\mathcal{O}_{\widetilde{\mathcal{S}}}(\widetilde{\mathcal{C}})$ to the fibers of $\mathcal{S} \longrightarrow \Delta$ is given by
    \[ \varepsilon_*\mathcal{O}_{\widetilde{\mathcal{S}}}(\widetilde{\mathcal{C}})_{|\mathcal{S}_o} \cong \varepsilon_{o,*}\mathcal{O}_{\widetilde{S}_o}(\widetilde{C}_o), \qquad  \varepsilon_*\mathcal{O}_{\widetilde{\mathcal{S}}}(\widetilde{\mathcal{C}})_{|\mathcal{S}_t} \cong \mathcal{O}_{\mathcal{S}_t}(\mathcal{C}_t). \]
    In particular, if $h^0(S_o,\varepsilon_{o,*}\mathcal{O}_{\widetilde{S}_o}(\widetilde{C}_o)) = h^0(S_t,\mathcal{O}_{\mathcal{S}_t}(\mathcal{C}_t))$ for every $t\in \Delta$, then sections of $\varepsilon_{o,*}\mathcal{O}_{\widetilde{S}_o}(\widetilde{C}_o)$ deform to sections of $\mathcal{O}_{\mathcal{S}_t}(\mathcal{C}_t)$ on nearby fibers.
\end{proposition}

To prove this, we need information on the local shape of $\mathcal{S}$ and $\mathcal{C}$ around the blown-up points. This is explained in the following 

\begin{lemma}\label{lemma:localequation}
	In the hypotheses of Proposition \ref{prop:smoothing}, the threefold $\mathcal{S}$ is smooth away from the points of $Z_1,Z_2$, where it has local equation
	\[ \{xy-tz=0\}\subset \cc_{x,y,z}\times \Delta. \]
	where $S_1 = \{ x=0 \}$ and $S_2 = \{ y=0 \}$.
	Moreover, $\mathcal{C}$ has local equations $\{ x = z =0 \}$ and $\{ y =z = 0 \}$ around the points of $Z_1$ and $Z_2$ respectively.
\end{lemma}

\begin{proof}
Recall the following deformation theoretic sheaves and vector spaces (see \cite{Fr, Ha}):
\[ \mathcal{T}_{S_o}^i=\mathcal{E}xt^i_{\mathcal{O}_{S_o}}(\Omega_{S_o}^1,\mathcal{O}_{S_o}), \hspace{1cm}\hbox{and}\hspace{1cm}\mathbb{T}^i_{S_o}={\rm{Ext}}^i_{\mathcal{O}_{S_o}}(\Omega_{S_o}^1,\mathcal{O}_{S_o}). \]  
The sheaf $\mathcal{T}^1_{S_o}$ parameterizes local first order deformations and $\mathcal{T}^2_{S_o}$ local obstructions. The vector space $\mathbb{T}_{S_o}^1$ parameterizes global first order deformations and $\mathbb{T}_{S_o}^2$ global obstructions. Recall that complete intersections are unobstructed so first order deformations can be lifted to a flat family over a small pointed disc $(\Delta, o)$, cf. \cite[Tag 0DZG]{S-P}. The local-to-global spectral sequence induces an exact sequence
\begin{equation}
\label{eq:loc-to-glob}
0\to H^1(T_{S_o})\to \mathbb{T}_{S_o}^1\to H^0(\mathcal{T}_{S_o}^1)\to H^2(T_{S_o}),
\end{equation}
where $T_{S_o}=\mathcal{T}^0$ is the tangent sheaf of $S_o$, and $H^1(T_{S_o})$ parameterizes locally trivial first order deformations. The sheaf $\mathcal{T}_{S_o}^1$ is supported on $B$ and isomorphic to the line bundle $\mathcal{N}_{B/S_1}\otimes \mathcal{N}_{B/S_2}$. The family $\mathcal{S}\to \Delta$ induces an element in $\mathbb{T}_{S_o}^1$ with image under the map in \eqref{eq:loc-to-glob} given by $f\in H^0\left(B, \mathcal{N}_{B/S_1}\otimes \mathcal{N}_{B/S_2}\right)$. The local equation of $\mathcal{S}$ around the central fiber is given by 
$$\left\{xy=tf(z)\right\}\subset \mathbb{C}_{x,y,z}\times\Delta.$$
The $3$-fold $\mathcal{S}$ is singular at the zeroes of $f$. On the other hand 
$$\mathcal{T}_{C_o}^1\cong\mathcal{O}_{(C_{o})_{sing}}\cong\mathcal{O}_W\hspace{.5cm}\hbox{and}\hspace{.5cm}\mathcal{T}_{S_o}^1\otimes\mathcal{O}_{C_o}\cong \mathcal{O}_{C_o\cap B}.$$
The restriction of the standard normal bundle sequence of $S_o$ to $C_o$ induces the map 
$$\mathcal{T}_{C_o}^1\to\mathcal{T}_{S_o}^1\otimes\mathcal{O}_{C_o}$$
given by the natural inclusion $\mathcal{O}_W\to \mathcal{O}_{C_o\cap B}$, cf. \cite[Thm. 3.11]{BV}. The family $\mathcal{C}\to \Delta$ is induced by an element in $\mathbb{T}_{C_o}^1$, whose image under the composition
\begin{equation}
\label{eq:def1}
\mathbb{T}_{C_o}^1\to H^0(\mathcal{T}_{C_o}^1)\to H^0(\mathcal{T}_{S_o}^1\otimes\mathcal{O}_{C_o})
\end{equation}
coincides with the image of $f\in H^0(\mathcal{T}_{S_o}^1)$ under the restriction map. The family $\mathcal{C}$ in $\mathcal{S}$ has local equations given by $\left\{xy=tf(z), z=0\right\}$ around the points $W$, and by $\left\{xy=tf(z), z=0, x=0\right\}$ and $\left\{xy=tf(z), z=0, y=0\right\}$ around the points of $Z_1$ and  $Z_2$ respectively. 

Note that the smoothness assumption on $\mathcal{C}$ forces $f$ to be nonzero at $W$, whereas the condition that the image of $f$ in $H^0(\mathcal{T}_{S_o}^1\otimes\mathcal{O}_{C_o})$ must lie in the image of \eqref{eq:def1} forces $f$ to vanish at $Z_1,Z_2$. This makes the total family $\mathcal{S}$ singular at the points $Z_i$. This does not come as a surprise, since the smooth divisor $\mathcal{C}$ restricted to the central fiber of $\mathcal{S}$ fails to be Cartier. 

To conclude we observe that $f$ has simple zeroes at the points of $Z_1$ and $Z_2$, otherwise the surface $S_t = \{ xy - f(z)t = 0 \}$ would not be smooth for $t\ne 0$. 
\end{proof} 

With this we can prove Proposition \ref{prop:smoothing}.

\begin{proof}[Proof of Proposition \ref{prop:smoothing}]:  
Since the statement is essentially local, we can restrict to a small neighborhood of $\mathcal{S}$ around a point $p\in Z_1$, but we keep the same notation. Thanks to the local description of Lemma \ref{lemma:localequation} we can compute that the blown-up threefold $\varepsilon\colon \widetilde{\mathcal{S}} \to \mathcal{S}$ is smooth, so that $\widetilde{\mathcal{C}}$ is a Cartier divisor on it.  The exceptional divisor of the blow-up is the smooth quadric surface $E\cong \mathbb{P}^1 \times \mathbb{P}^1$ and the central fiber of $\widetilde{\mathcal{S}} \to \Delta$ consists of the normal crossing surface $\widetilde{S}_o \cup E = \widetilde{S}_1\cup \widetilde{S}_2 \cup E$. We set $E_i = \widetilde{S}_i\cap E$: inside the smooth quadric $E$ these correspond to two intersecting lines. Instead $C_E := \widetilde{C}\cap E$ is another line, in the same ruling as $E_2$.   

Now we fix the line bundle $\mathcal{L}:=\mathcal{O}_{\widetilde{\mathcal{S}}}\left(\widetilde{\mathcal{C}}\right)$. It is clear that $\varepsilon_*\mathcal{L}_{|\mathcal{S}_t} \cong \mathcal{O}_{\mathcal{S}_t}(\mathcal{C}_t)$ for $t\ne 0$. We need to prove that $\varepsilon_*\mathcal{L}_{|\mathcal{S}_o} \cong \varepsilon_{o,*}\mathcal{O}_{\widetilde{\mathcal{S}}_o}(\widetilde{\mathcal{C}}_o)$.
To do so,  consider the diagram
\begin{equation}
\label{eq:sq.diag}
\begin{tikzcd}
\widetilde{S}_o\arrow[hook, r]\arrow[dr, "\varepsilon_o"']&\widetilde{S}_o\cup E\arrow[r, "i"]\arrow[d, "\mu"]&\widetilde{\mathcal{S}}\arrow[d, "\varepsilon"]&\widetilde{\mathcal{C}}\arrow[d, "\delta"]\arrow[hook', l, "v"']\\
&S_o\arrow[r, "j"]&\mathcal{S}&\mathcal{C}\arrow[hook', l, "u"'].
\end{tikzcd}
\end{equation}
First we are going to show that
\begin{equation}\label{eq:derivediso}
 j^*\varepsilon_* \mathcal{L} \cong \mu_* i^* \mathcal{L}
\end{equation}  
We observe that the pair of morphisms $S_1\cup S_2\to \mathcal{S}$, and $\widetilde{\mathcal{S}}\to \mathcal{S}$ are $\rm{Tor}$-independent, therefore we have an isomorphism in the derived category $Lj^*R\varepsilon_*\mathcal{L} \cong R\mu_* Li^*\mathcal{L}$
cf. \cite[Thm. 3.10.3]{Li}. We see that $Li^*\mathcal{L}\cong i^*\mathcal{L}$ since $\mathcal{L}$ is locally free. Then, to obtain \eqref{eq:derivediso}, it is enough to show that 
\begin{equation}\label{eq:derivediso2}
R^1\varepsilon_*\mathcal{L} = R^2\varepsilon_*\mathcal{L} = 0.
\end{equation} 
Indeed, in this case we have an isomorphism $Lj^*\mathcal{L} \cong R\mu_*\mathcal{L}$ in the derived category, and taking cohomology in degree zero we get exactly \eqref{eq:derivediso}.
To prove \eqref{eq:derivediso2}, consider the exact sequence
$$0\to \mathcal{O}_{\widetilde{\mathcal{S}}}\to\mathcal{L}\to v_* \mathcal{O}_{\widetilde{\mathcal{C}}}(\widetilde{\mathcal{C}})\to 0$$
We know that $R^1\varepsilon_*\mathcal{O}_{\widetilde{\mathcal{S}}} = R^2\varepsilon_*\mathcal{O}_{\widetilde{\mathcal{S}}} = 0$ since it is the blow up of a rational singularity. From the commutativity of the right square in \eqref{eq:sq.diag} one concludes that
$$R^i\eps_*\mathcal{L}\cong R^i\eps_*\left(v_*\mathcal{O}_{\widetilde{\mathcal{C}}}\left(\widetilde{\mathcal{C}}\right)\right)\cong u_*\left( R^i\delta_* \mathcal{O}_{\widetilde{\mathcal{C}}}\left(\widetilde{\mathcal{C}}\right)\right)\hspace{.5cm}\hbox{for}\hspace{.5cm}i=1,2.$$
Since we restricted ourselves to a small neighborhood of $p\in Z_1$ in $\mathcal{S}$, we can assume that both $\omega_\mathcal{C}$ and $u^*\omega_{S}$ are trivial. Then by adjunction we get
$$\mathcal{O}_{\widetilde{\mathcal{C}}}(\,\widetilde{\mathcal{C}}\,)\cong \mathcal{O}_{\widetilde{\mathcal{C}}}\left(C_E-v^* E\right)\cong\mathcal{O}_{\widetilde{\mathcal{C}}}$$ 
and since $\delta$ is the blow up of a smooth surface at a smooth point we see that $R^i\delta_{*}\mathcal{O}_{\widetilde{\mathcal{C}}}\left(\widetilde{\mathcal{C}}\right)=0$ for $i=1,2$. This shows that \eqref{eq:derivediso2} and consequently \eqref{eq:derivediso} hold.

Now we need to show that $\mu_*i^*\mathcal{L} \cong \varepsilon_{o,*}\mathcal{O}_{\widetilde{S}_o}(\widetilde{C}_o)$.  We see that $\widetilde{\mathcal{C}}$ and the central fiber $\widetilde{S}_o\cup E$ meet transversely, so that $i^*\mathcal{L}$ fits into the exact sequence
$$0\to i^*\mathcal{L}\to \mathcal{O}_{\widetilde{S}_o}(\widetilde{C}_o)\oplus\mathcal{O}_{E}(C_E)\to \mathcal{O}_{E_1\cup E_2}(C_E)\to 0.$$
pushing forward along  $\mu$ we get the sequence:
$$0\to \mu_*i^*\mathcal{L}\to \varepsilon_{o,*}\mathcal{O}_{\widetilde{S}_o}(\widetilde{C}_o)\oplus \left( H^0(E,\mathcal{O}_{E}(C_E))\otimes \mathcal{O}_p \right) \to H^0(E_1\cup E_2,\mathcal{O}_{E_1\cup E_2}(C_E))\otimes \mathcal{O}_p$$
and since $C_E$ is a line on $E$ and $E_1\cup E_2$ is the union of two lines in different rulings, it is straightforward to check that the restriction map
\[ H^0(E,\mathcal{O}_{E}(C_E)) \to H^0(E_1\cup E_2,\mathcal{O}_{E_1\cup E_2}(C_E)) \]
is an isomorphism, so that the last sequence shows that $\mu_*i^*\mathcal{L} \cong  \varepsilon_{o,*}\mathcal{O}_{\widetilde{S}_o}(\widetilde{C}_o)$.

The last part of Proposition \ref{prop:smoothing} follows from Grauert's theorem, since the sheaf $\varepsilon_*\mathcal{O}_{\widetilde{\mathcal{S}}}(\widetilde{\mathcal{C}})$ is torsion-free, hence flat over $\Delta$. 
\end{proof}

Now we can apply Proposition \ref{prop:smoothing} to obtain good pencils in genus 13, 14 and 15.

\section{Pointed curves in genus $13,14$ and $15$}\label{sec3}
\subsection{Genus $13$}

We recall the construction of a good pencil satisfying condition $(\star\star)$ for $\overline{\mathcal{M}}_{13,3}$. We follow the construction of \cite[Proof of Proposition 3.7]{Be}. Let $[C,p_1,p_2,p_3]\in \mathcal{M}_{13,3}$ be a general curve. Then 
\[ \dim W^{2}_{11}(C) = \rho(13,2,11) = 1 \]
so that there exists a line bundle $L$ which is a complete $\mathfrak{g}^2_{11}$ that does not separate the points $p_1,p_2,p_3$. Furthermore, when such $(C,p_1,p_2,p_3)$ and $L$ are general, the residual $K_C-L$ is a very ample $\mathfrak{g}^3_{13}$ that embeds $C$ as a curve in $\mathbb{P}^3$ such that $\dim |\mathcal{I}_C(5)|=2$. Let $S\in|\mathcal{I}_C(5)|$ be a general quintic surface containing $C$. The surface $S$ is \textit{canonical}, i.e., $K_S = \mathcal{O}_S(H)$, where $H$ is the hyperplane class of $\mathbb{P}^3$. Then  $\mathcal{O}_C(C) = L$ so that $h^0(S,\mathcal{O}_S(C)) = 4$. Moreover, since $L$ does not separate the points $p_1,p_2,p_3$ we see that
\[ \Gamma = |\mathcal{O}_S(C)\otimes \mathcal{I}_{p_1,p_2,p_3}| \]
is a non-isotrivial pencil. Now we need to prove that $\Gamma$ satisfies condition $(\star\star)$.  
\begin{lemma}\label{lemma:genus13pencil} 
	For a general choice of the curve $(C,p_1,p_2,p_3)$, the linear series $L \in W^2_{11}(C)$, and quintic $S$ as before, the pencil $\Gamma$ satisfies condition $(\star\star)$. 
\end{lemma}
\begin{proof}
	We first prove that $S$ is smooth and that all curves in $\Gamma$ are at worst nodal. Indeed, in this case the curves are automatically stable, because the corresponding dualizing sheaf is given by the restriction to the curve of $\mathcal{O}_S(C+K_S) \cong \mathcal{O}_S(C+H)$ and since $\mathcal{O}_S(C)$ is globally generated and $\mathcal{O}_S(H)$ is very ample, we see that the dualizing sheaf is ample.
	
	We check this by an explicit degeneration, with the help of the Macaulay2 code \cite{MacaulayCode}, which has been inspired by \cite{KT}. As usual in these cases, the explicit computations are done over a finite field $\mathbb{F}_p$, but by semicontinuity 
	an explicit example over $\mathbb{F}_p$ shows that there is another example over $\mathbb{Q}$ and then $\mathbb{C}$.    
	
	Benzo considers in \cite[Proof of Proposition 3.7]{Be} a degeneration where the quintic is the union of a smooth quartic and a plane, instead we consider a degeneration where the quintic is the union of a quadric and a cubic, because when both surfaces are rational the computations can be made explicit. Thus, let $S_1\subseteq \mathbb{P}^3$ be a smooth cubic surface and $S_2\subseteq \mathbb{P}^3$ a smooth quadric that intersect $S_1$ transversally along a canonical genus four curve $B=S_1\cap S_2$. We can look at $S_1$ as the blow up of $\mathbb{P}^2$ at $6$ points, embedded by the linear system $-K_{S_1} \sim 3L - E$, where $L$  is the class of a line in $\mathbb{P}^2$ and $E$ is the exceptional divisor of the blow-up.
	
	Let $S_o=S_1\cup S_2$ as in Section \ref{sec2}. We can choose on $S_2$ a general curve $C_2$ of class $\mathcal{O}_{S_2}(3,1)$: this is a rational curve of degree $4$ and intersects $B$ transversally along $12$ points. Amongst these, we can choose a subset $W$ of $8$ points and then take a general curve $C_1\subseteq S_1$ of class $13L-5E$ passing through $W$ : this is a smooth curve of genus $6$ and degree $9$ in $\mathbb{P}^3$, and the union $C_o=C_1\cup C_2$ has arithmetic genus $13$ and degree $13$. 
	
	Since the Hilbert scheme of genus 13 and degree 13 curves in $\mathbb{P}^3$ is irreducible, $C_o$ can be smoothed in a smooth family $\mathcal{C}$: cf. \cite[Proof of Proposition 3.7]{Be}. Furthermore, it is easy to check via Macaulay2 that $h^0(\mathbb{P}^3,\mathcal{I}_{C_o}(5))=3$ and that a general quintic that contains $C_o$ is smooth, so that $S_o$ can be smoothed out to a family $\mathcal{S}$. Then, we are in the setting of Proposition \ref{prop:smoothing}, and to show that sections of $\varepsilon_{o,*}\mathcal{O}_{\widetilde{S}_o}(\widetilde{C}_o)$ smooth to sections of $\mathcal{O}_{S_t}(C_t)$ we need to prove that $h^0(S_o,\varepsilon_{o,*}\mathcal{O}_{\widetilde{S}_o}(\widetilde{C}_o))=4$. By construction, $H^0(S_o,\varepsilon_{o,*}\mathcal{O}_{\widetilde{S}_o}(\widetilde{C}_o))$ is the kernel of the difference map
	\[
	H^0(S_1,\mathcal{O}_{S_1}(C_1)\otimes \mathcal{I}_{Z_1,S_1} ) \oplus H^0(S_2,\mathcal{O}_{S_2}(C_2)\otimes \mathcal{I}_{Z_2,S_2}) \longrightarrow H^0(B,\mathcal{O}_{B}(W))
	\]
	where we have kept the same notation as in \eqref{eq:pushforwardF}. Via Macaulay2 \cite{MacaulayCode} we can easily compute 
	\[ h^0(S_1,\mathcal{O}_{S_1}(C_1)\otimes \mathcal{I}_{Z_1,S_1} )=5, \qquad h^0(S_2,\mathcal{O}_{S_2}(C_2)\otimes \mathcal{I}_{Z_2,S_2}) = 4, \qquad h^0(B,\mathcal{O}_{B}(W))=5 \]
	and moreover we can check that the restriction map $H^0(S_1,\mathcal{O}_{S_1}(C_1)\otimes \mathcal{I}_{Z_1,S_1} )\to H^0(B,\mathcal{O}_B(W))$ is an isomorphism and that the map $H^0(S_2,\mathcal{O}_{S_2}(C_2)\otimes \mathcal{I}_{Z_2,S_2} )\to H^0(B,\mathcal{O}_B(W))$ is injective.  Hence the kernel of the difference map above has the expected dimension $4$. Hence, Proposition \ref{prop:smoothing} applies, so that sections in $H^0(S_o,\varepsilon_{o,*}\mathcal{O}_{\widetilde{S}_o}(\widetilde{C}_o))$ extend to sections in $H^0(S_t,C_t)$. 
	
	Then, we can compute explicitly a pencil $\Gamma$ in $H^0(S_o,\varepsilon_{o,*}\mathcal{O}_{\widetilde{S}_o}(\widetilde{C}_o))$ where all curves are stable, and by the previous discussion this can be smoothed out to a pencil with property $(\star\star)$. To compute this pencil, we start with a general pencil $\Gamma_2 \subseteq |\mathcal{O}_{S_2}(C_2)\otimes \mathcal{I}_{Z_2,S_2}|$: we can check that the points of $Z_2$ are general in $S_2$ so that all curves in $\Gamma_2$ are at worst nodal with at most one node. Moreover, all the nodes are away from the curve $B$. Via Macaulay2 \cite{MacaulayCode} we can then compute a pencil $\Gamma_1 \subseteq |\mathcal{O}_{S_1}(C_1)\otimes \mathcal{I}_{Z_1,S_1}|$ whose restriction to $|\mathcal{O}_B(W)|$ coincides with that of $\Gamma_2$. This gives a pencil $\Gamma$ in $H^0(S_o,\varepsilon_{o,*}\mathcal{O}_{\widetilde{S}_o}(\widetilde{C}_o))$ and we can compute with Macaulay2 \cite{MacaulayCode} that all curves in $\Gamma_1$ are also at worst nodal and that the nodes are away from $B$.
	
	To conclude, we need to check the last condition of $(\star\star)$, namely that the pencil $\Gamma$ does not intersect the boundary divisor $\Delta_{1:\varnothing}$. First we observe that, since the pencil $\Gamma_2$ is general, every element there is either smooth or an union of two smooth components meeting at one node.  Then we are going to prove that every curve in the pencil $\Gamma_1$ is irreducible and with at most one node. With the help of Macaulay2 \cite{MacaulayCode} we can prove that all the curves in the pencil $\Gamma_1$ have at most one node.  Now we see from \cite{dR} that $C_1$ is very ample on $S_1$, and then \cite[Theorem A]{BL} shows that any curve in $\Gamma_1$ is 2-connected and in particular it is not the union of two components meeting at one node, thus every curve in $\Gamma_1$ is irreducible.
	
	
	The argument above shows that every curve in the resulting pencil $\Gamma$ on $S_1\cup S_2$ has one of the following \textit{dual graphs} (see \cite{ACG}):

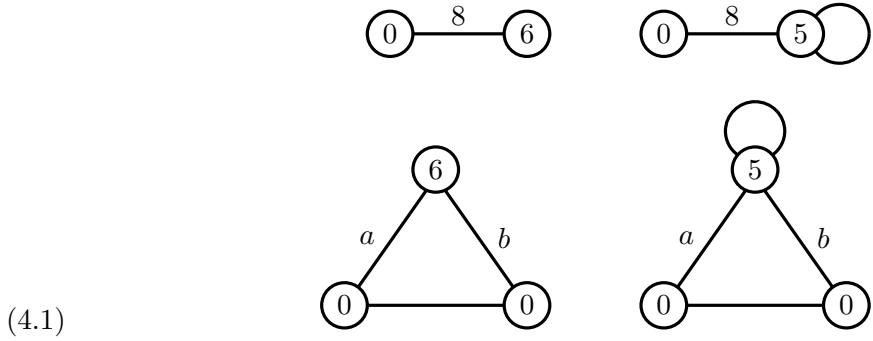
\begin{equation}\label{dual.graph}
\begin{tikzpicture}[scale=0.6]
\draw[very thick] (0,0) circle (0.50);
\node at (0,0){$0$};
\node at (1.5,0.4){$8$};
\draw[very thick] (0.5,0) -- (2.5,0);

\draw[very thick] (3,0) circle (0.50);
\node at (3,0){$6$};

\draw[very thick] (6,0) circle (0.50);
\node at (6,0){$0$};
\node at (7.5,0.4){$8$};
\draw[very thick] (6.5,0) -- (8.5,0);

\draw[very thick] (9,0) circle (0.50);
\node at (9,0){$5$};

\draw[very thick] (10.5,0) arc (0:142:0.65);
\draw[very thick] (10.5,0) arc (0:-142:0.65);

\draw[very thick] (1,-3) circle (0.50);
\node at (1,-3){$6$};
\node at (-0.5,-4.5){$a$};
\node at (2.5,-4.5){$b$};
\draw[very thick] (0.80,-3.45) -- (-0.70,-5.6);
\draw[very thick] (1.20,-3.45) -- (2.70,-5.6);

\draw[very thick] (-1,-6) circle (0.50);
\node at (-1,-6){$0$};
\draw[very thick] (-0.5,-6) -- (2.5,-6);

\draw[very thick] (3,-6) circle (0.50);
\node at (3,-6){$0$};

\draw[very thick] (8,-3) circle (0.50);
\node at (8,-3){$5$};
\node at (6.5,-4.5){$a$};
\node at (9.5,-4.5){$b$};
\draw[very thick] (8,-1.5) arc (90:235:0.65);
\draw[very thick] (8,-1.5) arc (90:-55:0.65);
\draw[very thick] (7.80,-3.45) -- (6.3,-5.6);
\draw[very thick] (8.20,-3.45) -- (9.70,-5.6);

\draw[very thick] (6,-6) circle (0.50);
\node at (6,-6){$0$};
\draw[very thick] (6.5,-6) -- (9.5,-6);

\draw[very thick] (10,-6) circle (0.50);
\node at (10,-6){$0$};

\end{tikzpicture}
\end{equation}

The number on the edges represent how many edges are there between the adjacent vertices and we omit it if the number of edges is one. Moreover, $a$ and $b$ are non negative integers subject to the condition $a+b=8$. Recall \cite{ACG} that if $G$ is a dual graph and $e$ an edge between vertices $v_1,v_2$, a \textit{smoothing} of $G$ along $e$ is the dual graph obtained from $G$ by removing $e$ and replacing $v_1,v_2$ with a single vertex $v$ with $g(v)=g(v_1)+g(v_2)$. Observe that if there is more than one edge between $v_1$ and $v_2$, smoothing along one forces the remaining edges to become self-edges at $v$. If $e$ is a self edge on $v$, then a smoothing along $e$ consist of removing $e$ and increasing $g(v)$ by one. A \textit{smoothing} of a dual graph is an iteration of the described surgery on edges. Let $C$ be a curve with dual graph $G$ and let $H$ be another dual graph. Then $[C]\in\overline{\mathcal{M}}_g$ lies in the boundary strata $\Delta_{H}$ if $H$ is a \textit{smoothing} of $G$. One can easily see from \eqref{dual.graph} that no curve on $\Gamma$ lies in $\Delta_i$, for $i=1, \ldots, 6$.  In particular, the last condition of $(\star\star)$ is fulfilled.

\end{proof}

\begin{proposition}
\label{propM13,4}
	The moduli space $\overline{\mathcal{M}}_{13,4}$ is uniruled.
\end{proposition}
\begin{proof}
	We use the principle described in Proposition \ref{prop2.sec1}. Let $[C,p_1,p_2,p_3]$ be a general point in $\mathcal{M}_{13,3}$ and $\Gamma$ a good pencil as constructed in Lemma \ref{lemma:genus13pencil}. Observe that in order to resolve the pencil we have to blow up the quintic surface $S$ at $11$ points. Let $D$ be a quintic plane curve $D\in |\mathcal{O}_S(1)|$.  
	\begin{equation}
	\label{eq1.sec2}
	\begin{tikzcd}
	\mathcal{X}_D\arrow[d]\arrow[r]& \tilde{S}\arrow[d]\arrow[r, "Bl_{11}"]&S\\
	D\arrow[r]&\Gamma.&
	\end{tikzcd}
	\end{equation}
	Recall that $C\subset \pp^3$ has degree $13$ and $S$ is a canonical surface, giving us 
	$$(C\cdot D)=13\quad\hbox{and}\quad (K_S\cdot D)=5.$$ 
Moreover, one computes that 	
	$$\chi_{top}\left(\tilde{S}\right)=66\quad \hbox{and}\quad \chi(\mathcal{O}_{\tilde{S}})=5.$$
The formulas \eqref{standard} give us
	$$ (\Gamma\cdot K_{\overline{\mathcal{M}}_{13,3}})=-4.$$
	Then by Proposition \ref{prop2.sec1} we conclude that $\overline{\mathcal{M}}_{13,4}$ is uniruled.
\end{proof}

\subsection{Genus $14$}\label{sec:genus14}
Pencils in genus $14$ rely on the existence of a smooth $(2,2,2,2)$ complete intersection surface in $\mathbb{P}^6$ containing the curve in question. We will construct pencils satisfying $(\star\star)$ by smoothing pencils in the union of two rational surfaces meeting along a canonical curve of genus $7$ in $\pp^6$. This is inspired by \cite{V} and \cite{BV}. We borrow the description of a general curve of genus $14$ from \cite{V}. A general curve $C$ of genus $14$ admits finitely many line bundles $L$ of degree $8$ such that $h^0(C,L)=2$. For every such line bundle the residual $K_C-L$ is very ample and induces an embedding 
$C\subseteq\pp^6$  
realizing $C$ as a degree $18$ curve in a six dimensional projective space. Moreover, the space $H^0(\pp^6, \mathcal{I}_C(2))$ of quadrics vanishing at $C$ is five dimensional. In particular $C$ lies in a $(2,2,2,2)$ complete intersection surface 
$$C\subset S=Q_1\cap\ldots\cap Q_4.$$
The surface $S$ is canonical, hence  by adjunction we see that $\mathcal{O}_C(C) \cong L$, the $\mathfrak{g}^1_8$ that we started with. In particular, we see that 
$\dim |\mathcal{O}_S(C)|=2$
and for a general point $p\in S$, the linear system consisting of curves in $|\mathcal{O}_S(C)\otimes \mathcal{I}_p|$ is a general pencil $\Gamma$ in $|\mathcal{O}_S(C)|$.

Now we should check that $\Gamma$ satisfies condition $(\star\star)$. The idea is to degenerate to a reducible surface $S_o=S_1\cup S_2$ as above. We summarize here the construction in \cite[Section 3]{BV}: let $S_1\subset \pp^6$ be a degree $6$ Del Pezzo surface and $Q_1,\ldots,Q_4$ four general quadrics containing $S_1$. Recall that $\mathcal{I}_{S_1/\pp^6}$ is generated by quadrics. By \cite[Lemma 3.2]{BV} and the subsequent discussion, the scheme theoretic intersection of the quadrics is reduced and
$$Q_1\cap\ldots\cap Q_4=S_1\cup S_2,$$
where $S_2$ is a smooth irreducible rational surface of degree $10$. Moreover, $S_1$ and $S_2$ meet transversally along a quadric section $B=S_1\cap S_2\in|\mathcal{O}_{S_1}(2H)|$. The curve $B$ is a smooth canonical curve of genus $7$ and degree $12$ in $\pp^6$. By the discussion after \cite[Lemma 3.2]{BV} the surface $S_2$ is isomorphic to the blow-up of $\pp^2$ at general points $p_1,\ldots,p_{11}$ and the hyperplane class $H$ is given by $6L-2(E_1+\dots+E_5)-(E_6+\dots+E_{11})$, where $L$ is the class of a line in $\mathbb{P}^2$ and the $E_i$ are the exceptional divisors of the blow-up. Instead, \cite[Equation (3.6)]{BV} shows that the class of $B$ in $S_2$ is given by $H-K_{S_2}$. A general curve $C_2\in|2L-E_1-E_2-E_{10}-E_{11}|$ is a smooth rational normal curve in $\pp^6$. Finally, we fix the intersection $W:=C_2\cap B$, which consists of $8$ distinct points, by \cite[Proposition 3.8]{BV}, and we take a general element 
\begin{equation}\label{eq:C1genus14}
C_1\in\left|\mathcal{O}_{S_1}\left(B\right)\otimes \mathcal{I}_{W/S_1}\right|.
\end{equation} 
Now, let $S_o=S_1\cup S_2$ and $C_o=C_1\cup C_2$. The curve $C_o\subseteq \pp^6$ has degree $18$ and arithmetic genus $14$ and the pair $(S_o,C_o)$ smooths out as in Proposition \ref{prop:smoothing}: 

\begin{lemma}
	\label{lemma1.gen14}
	The pair $(S_o,C_o)$ smooths out to families $\mathcal{C}\subseteq \mathcal{S}$ as in Proposition \ref{prop:smoothing}, where $\mathcal{S}$ is a family of $(2,2,2,2)$ complete intersections. In particular, when $C\in \mbar_{14}$ and $L\in W^1_8(C)$ are general, the general complete intersection surface $S$ as above is smooth. 
\end{lemma} 

The proof is completely analogous to the proof of \cite[Theorem 3.11]{BV} in genus $15$. We summarize the construction for sake of completeness.  

\begin{proof}[Proof of Lemma \ref{lemma1.gen14}]
	Consider the exact sequences
	$$0\to T_{\pp^6}\otimes\mathcal{O}_{C_o}\to T_{\pp^6}\otimes\left(\mathcal{O}_{C_1}\oplus\mathcal{O}_{C_2}\right)\to T_{\pp^6}\otimes\mathcal{O}_W\to0$$
	$$0\to\mathcal{O}_{C_2}\to\mathcal{O}_{C_2}(H)^{\oplus 7}\to T_{\pp^6}\otimes\mathcal{O}_{C_2}\to0,$$
	where the second one is the restriction of the Euler sequence to $C_2$. The map $H^0(T_{\pp^6}\otimes\mathcal{O}_{C_2})\to H^0(T_{\pp^6}\otimes\mathcal{O}_W)$ is an isomorphism and $h^1(T_{\pp^6}\otimes\mathcal{O}_{C_1})=0$, since $C_1$ is a canonical curve. This implies $h^1(T_{\pp^6}\otimes\mathcal{O}_{C_o})=0$ ( see \cite[Proof of Prop. 2.6]{BV} for details) placing $C_o\hookrightarrow \pp^6$ in the unique component $\mathcal{H}$ of the Hilbert scheme of curves of genus $14$ and degree $18$ in $\pp^6$ that dominates $\overline{\mathcal{M}}_{14}$. Moreover, the Kodaira-Spencer map 
	$$T_{\left[C_o\hookrightarrow\pp^6\right]}\mathcal{H}\longrightarrow T_{\left[C_o\right]}\overline{\mathcal{M}}_{14}$$
	is surjective and $C_o$ can be smoothed as a degree $18$ curve of genus $14$ inside $\pp^6$, in such a way that the total family $\mathcal{C} \to \Delta$ is smooth.
	
	Let $N\subset S_1$ be one of the $(-1)$-curves in $S_1$. The curves $C_o$ and $N$ intersect transversally at $2$ points not in $W$. Moreover, $h^0(\mathbb{P}^6,\mathcal{I}_{S_1\cup C_2/\mathbb{P}^6}(2H))=h^0(\mathbb{P}^6,\mathcal{I}_{C_o\cup N/\mathbb{P}^6}(2H))=4$ and $h^i(\mathbb{P}^6,\mathcal{I}_{C_o\cup N/\mathbb{P}^6}(2H))=0$ for $i\geq 1$, cf. \cite[Prop. 3.10]{BV}. From the exact sequence
	$$0\longrightarrow \mathcal{I}_{C_o\cup N/\mathbb{P}^6}(2)\longrightarrow\mathcal{I}_{C_o/\mathbb{P}^6}(2)\longrightarrow\mathcal{O}_N\longrightarrow 0$$ 
	one concludes that $h^0(\mathbb{P}^6,\mathcal{I}_{C_o/\mathbb{P}^6}(2))=5$ and $h^i(\mathbb{P}^6,\mathcal{I}_{C_o/\mathbb{P}^6}(2))=0$ for $i\geq 1$. Let 
	\begin{equation}
	\label{eq1:lemma1.gen14}
	\widetilde{\mathcal{H}}\longrightarrow \mathcal{H}
	\end{equation}
	be the Hilbert scheme parameterizing flags in $\pp^6$ 
	$$\left[C\hookrightarrow S\hookrightarrow \pp^6\right]\in \widetilde{\mathcal{H}},$$
	where $\left[C\hookrightarrow \pp^6\right]\in \mathcal{H}$ and $S$ is a complete intersection of four independent quadrics containing $C$. The map \eqref{eq1:lemma1.gen14} is generically a $\pp^4$-bundle. Let $\mathcal{U}\subset \widetilde{\mathcal{H}}$ be the open subset where the map \eqref{eq1:lemma1.gen14} restricts to a $\pp^4$-bundle. The argument before places the flag $\left[C_o\hookrightarrow S_o\hookrightarrow \pp^6\right]$ in $\mathcal{U}$. The same argument as in \cite[Thm. 3.11]{BV} shows that for  general point in $\mathcal{U}$, the surface $S$ is smooth.
\end{proof}

Now we are in place to show the stability of every element on a general pencil.

\begin{lemma}
	\label{lemma2.gen14}
	For a general choice of curve $(C,p)\in\overline{\mathcal{M}}_{14,1}$, linear series $L \in W^1_{8}(C)$, and $(2,2,2,2)$ complete intersection surface $S$ as before, the pencil $\Gamma=\left|\mathcal{O}_{S}(C)\otimes\mathcal{I}_{p}\right|$ satisfies condition $(\star\star)$. 
\end{lemma}

\begin{proof}
	As in the proof of Lemma \ref{lemma:genus13pencil}, we start by showing that all curves in $\Gamma$ are at worst nodal.
	We keep the same notation as above. 
	Lemma \ref{lemma1.gen14} puts us in the setting of Proposition \ref{prop:smoothing}, and in order to show that sections of $\varepsilon_{o,*}\mathcal{O}_{\widetilde{S}_o}(\widetilde{C}_o)$ deform to sections of $\mathcal{O}_{S_t}(C_t)$, we have to show that 
	$$h^0\left(S_o,\varepsilon_{o,*}\mathcal{O}_{\widetilde{S}_o}(\widetilde{C}_o)\right)=3.$$
	Note that $Z_2=\varnothing$ and $Z_1=C_1\cap B\setminus W$ with $|Z_1|=16$. By construction, $H^0(S_o,\varepsilon_{o,*}\mathcal{O}_{\widetilde{S}_o}(\widetilde{C}_o))$ is the kernel of the map 
	\begin{equation}
	\label{eq2:lemma.g15}
	H^0(S_1,\mathcal{O}_{S_1}(C_1)\otimes \mathcal{I}_{Z_1,S_1} ) \oplus H^0(S_2,\mathcal{O}_{S_2}(C_2)) \longrightarrow H^0(B,\mathcal{O}_{B}(W)).
	\end{equation}
	By the proof of  \cite[Proposition 3.8]{BV} , the restriction map $H^0(S_2,\mathcal{O}_{S_2}(C_2)) \longrightarrow H^0(B,\mathcal{O}_{B}(W))$ is an isomorphism and moreover $h^0(S_2,\mathcal{O}_{S_2}(C_2))=h^0(B,\mathcal{O}_B(W))=2$.
	Then, the exact sequence
	$$0\longrightarrow \mathcal{O}_{S_1}\longrightarrow \mathcal{O}_{S_1}(C_1)\otimes \mathcal{I}_{Z_1/S_1}\longrightarrow \mathcal{O}_{B}(W)\longrightarrow0$$ 
	shows that $h^0(S_1,\mathcal{O}_{S_1}(C_1)\otimes \mathcal{I}_{Z_1/S_1})=3$ and that the map 
	$
	H^0(S_1,\mathcal{O}_{S_1}(C_1)\otimes \mathcal{I}_{Z_1/S_1})\longrightarrow H^0(B, \mathcal{O}_{B}(W))
	$  
	is surjective.  Thus, the kernel of the map \eqref{eq2:lemma.g15} has dimension $3$ as expected  
	and we can apply Proposition \ref{prop:smoothing}.
	
	To conclude we construct a general pencil in $H^0\left(S_o,\varepsilon_{o,*}\mathcal{O}_{\widetilde{S}_o}(\widetilde{C}_o)\right)$. To do so, choose a general pencil  $\Gamma_1 \subseteq |\mathcal{O}_{S_1}(C_1)\otimes \mathcal{I}_{Z_1/S_1}|$: since the restriction from $H^0(S_1,\mathcal{O}_{S_1}(C_1)\otimes \mathcal{I}_{Z_1/S_1})$ to $H^0(B, \mathcal{O}_{B}(W))$ is surjective, this induces the complete pencil $\Gamma_{B}=|\mathcal{O}_B(W)|$ and since the restriction map from $H^0(S_2,\mathcal{O}_{S_2}(C_2))$ to $H^0(B,\mathcal{O}_{B}(W))$ is an isomorphism, this corresponds to the complete pencil $\Gamma_2 =|\mathcal{O}_{S_2}(C_2)|$. This builds a pencil $\Gamma$ on $S_o$ and to prove that the corresponding curves are at worst nodal, it is enough to show the same for the curves in $\Gamma_1$ and $\Gamma_2$ and we should also check that all the nodes are away from the curve $B$.  This is easily verified for $\Gamma_2$, since this is just a pencil of conics passing through 4 general points in $\mathbb{P}^2$. For $\Gamma_1$ we can prove with Macaulay2 \cite{MacaulayCode} that all the curves in this pencil are at worst one-nodal and that the nodes are away from $B$. At this point, we see from \cite{dR} that the linear system $|\mathcal{O}_{S_1}(C_1)|$ is very ample, and  then, it follows from \cite[Theorem A]{BL} that all the curves in $\Gamma_1$ are 2-connected, so they must be all irreducible.

	
	Finally, we should prove that the resulting pencil $\Gamma$ on $S_1\cup S_2$ intersects trivially $\Delta_1$ in $\overline{\mathcal{M}}_{14}$. As in the proof of Lemma \ref{lemma:genus13pencil}, by looking at the possible dual graphs one concludes that $\Gamma$  does not intersect $\Delta_i$ for $i=1,\ldots,7$.		
\end{proof}

As corollary we can prove:

\begin{proposition}
\label{propM14,3}
	The moduli space $\overline{\mathcal{M}}_{14,3}$ is uniruled.
\end{proposition}
\begin{proof}
	Recall that the following intersections hold in the canonical surface $S\subset \pp^6$:
	$$(C\cdot K_S)=18,\,\,\, (K_S^2)=16,\,\,\,\hbox{and}\,\,\, (C^2)=8.$$
	Let $\Gamma$ be the pencil $|\mathcal{O}_{S}(C)\otimes\mathcal{I}_p|$, where $p$ is a general point on $C$. To resolve the pencil $S\dashrightarrow \Gamma$ we have to blow up $S$ at $8$ points corresponding to the base locus of $\Gamma$. Observe that by \eqref{standard}, the following intersections hold in $\overline{\mathcal{M}}_{14,1}$:
	\begin{equation}
	(\Gamma\cdot\lambda)=21,\,\,\, (\Gamma\cdot\psi)=1,\,\,\,\text{and}\,\,\,(\Gamma\cdot\delta)=140.
	\end{equation}
	Let $D_1, D_1$ be two independent curves in $|\mathcal{O}_S(K_S)|$ away from the base locus of $\Gamma$.
	By Proposition \ref{prop2.sec1}, the curve $\Theta$ induced by base change
	\begin{equation}
	\begin{tikzcd}
	\widetilde{\mathcal{Y}}\arrow[d]\arrow[r]& \tilde{S}\arrow[d]\arrow[r, "Bl_{8}"]&S\\
	\Theta=D_1\times_\Gamma D_2\arrow[r]&\Gamma.&
	\end{tikzcd}
	\end{equation}
	defines a nef curve on $\overline{\mathcal{M}}_{14,3}$. Since 
	$$(\Theta\cdot K_{\overline{\mathcal{M}}_{14,3}})=18^2\left(-2+\frac{16}{9}\right)-16<0,$$
	Lemma \ref{lemma2.gen14} and Proposition \ref{prop2.sec1} show that the moduli space $\overline{\mathcal{M}}_{14,3}$ is uniruled.
\end{proof}

\subsection{Genus $15$}\label{sec.gen15}
The construction in genus $15$ relies on the same $(2,2,2,2)$ complete intersection surfaces $S\subseteq \mathbb{P}^6$ as in genus $14$. In particular, we will use the same normal crossing surface $S_o=S_1\cup S_2$ used in  genus $14$ and a nodal reducible curve $C_o=C_1\cup C_2$, the only difference is that now we define $C_1$ to be a general element
\begin{equation}\label{eq:C1genus15}
C_1\in\left|\mathcal{O}_{S_1}\left(B+N\right)\otimes \mathcal{I}_{W/S_1}\right|.
\end{equation} 
where $W:=C_2\cap B$ is as in Section \ref{sec:genus14} and $N$ is one of the $(-1)$-curves on $S_1$.

We recall the setting from  \cite{BV}. When $C$ is a general curve of genus $15$, the two Brill-Noether spaces 
$$ \dim W^1_{9}(C) = 1, \qquad \dim W^6_{19}(C) = 1  $$
are residual to each other. For a general $L\in W^1_{9}(C)$, the residual line bundle $K_C-L$ is very ample and induces an embedding $C\subseteq \mathbb{P}^6$ such that $h^0(\mathcal{I}_{C/\pp^6}(2))=4$. The intersection of all the quadrics containing $C$ is a complete intersection surface: 
$$ C\subseteq S=Q_1\cap Q_2\cap Q_3 \cap Q_4.$$
We observe that the surface is uniquely determined by $L\in W^1_9(C)$. 

The surface $S$ is canonical, meaning that $K_S \cong \mathcal{O}_S(H)$, where $H$ is the hyperplane class in $\mathbb{P}^6$. Then, adjunction implies that $\mathcal{O}_C(C)=\mathcal{O}_C(K_C-H)\cong L$ is the $\mathfrak{g}^1_{9}$ that we started with and moreover $h^0(S,\mathcal{O}_S(C))=3$, meaning that $C$ moves in net on $S$. Furthermore, since there is a $1$-dimensional family of $\mathfrak{g}_9^1$, given a general two-pointed curve $(C,p_1,p_2)\in \mathcal{M}_{15,2}$, we can choose an $L$ that does not separate $p_1,p_2$ so that the linear system
$$\Gamma=\left|\mathcal{O}_S(C)\otimes \mathcal{I}_{p_1,p_2}\right|$$
is one dimensional. Farkas and Verra in \cite[Proposition 4]{FV2} show that $\Gamma$ satisfies condition $(\star\star)$, moreover $\Gamma$ has only irreducible curves. 

\begin{lemma}[Proposition 4 in \cite{FV2}]
\label{lemma:genus15pencil} 
	For a general choice of the curve $(C,p_1,p_2) \in \mathcal{M}_{15,2}$ and the linear series $L \in W^1_{9}(C)$ as before the pencil $\Gamma$ satisfies condition $(\star\star)$. 
\end{lemma}

We add an alternative proof as an application of Proposition \ref{prop:smoothing}: in \cite[Thm. 3.11]{BV} Bruno and Verra prove that the surface $S$ is smooth by considering a degeneration to the reducible surface  $S_o = S_1\cup S_2$ of Section \ref{sec:genus14}. We consider pencils on the same surface in order to apply Proposition \ref{prop:smoothing}.

\begin{proof}[Proof of Lemma \ref{lemma:genus15pencil}]
	In \cite[Theorem 3.11]{BV} it is shown that the analogue of Lemma \ref{lemma1.gen14} holds in genus $15$ as well. Then the proof is completely analogous to that of Lemma \ref{lemma2.gen14}. We use the same notation here, keeping in mind that the curve $C_1$ is as in \eqref{eq:C1genus15}.

	We wish to apply Proposition \ref{prop:smoothing}, and to do so  we need to show that the kernel of the map
	\begin{equation}\label{eq2:lemma2.gen15}
	H^0(S_1,\mathcal{O}_{S_1}(C_1)\otimes \mathcal{I}_{Z_1/S_1} ) \oplus H^0(S_2,\mathcal{O}_{S_2}(C_2)) \longrightarrow H^0(B,\mathcal{O}_{B}(W)).
	\end{equation}
	has dimension three. We recall from the proof of Lemma \ref{lemma2.gen14} that the map 
	$$H^0(S_2,\mathcal{O}_{S_2}(C_2))\longrightarrow H^0(B,\mathcal{O}_B(W))$$
	is an isomorphism between two spaces of dimension $2$. Then, the exact sequence
	$$0\longrightarrow \mathcal{O}_{S_1}(N)\longrightarrow \mathcal{O}_{S_1}(C_1)\otimes \mathcal{I}_{Z_1/S_1}\longrightarrow \mathcal{O}_{B}(W)\longrightarrow0$$ 
	shows that $h^0(S_1,\mathcal{O}_{S_1}(C_1)\otimes \mathcal{I}_{Z_1/S_1})=3$ and that the map 
	$
	H^0(S_1,\mathcal{O}_{S_1}(C_1)\otimes \mathcal{I}_{Z_1/S_1})\longrightarrow H^0(B, \mathcal{O}_{B}(W))
	$  
	is surjective.  Thus, the kernel of the map \eqref{eq2:lemma2.gen15} has dimension $3$ as expected. Now, as in the proof of Lemma \ref{lemma2.gen14}, it is enough to show that every curve in a general pencil $\Gamma_1 \subseteq |\mathcal{O}_{S_1}(C_1)\otimes \mathcal{I}_{Z_1/S_1}|$ is at worst one-nodal and with all nodes away from $B$. This can also be done with Macaulay2 \cite{MacaulayCode}. Finally, by an analysis on possible dual graphs analogous to the proof of Lemma \ref{lemma:genus13pencil}, one concludes that the resulting pencil $\Gamma$ on $S_1\cup S_2$ does not intersect $\Delta_i$ for $i=1,\ldots,7$.
\end{proof}

\section{Kodaira dimension of $\overline{\mathcal{M}}_{12,n}$}
\label{sec5}

We turn our attention to pointed curves of genus $12$. The strategy is inspired by \cite{CR3,FV2}: we provide a bound by constructing nef curves that come from covering curves of $\overline{\mathcal{M}}_{11,n+2}$, pushed forward to $\overline{\mathcal{M}}_{12,n}$ via the $\Delta_{\rm{irr}}$-boundary map. This is the map
\begin{equation}
\label{boundary}
\theta = \theta_{g,n}:\overline{\mathcal{M}}_{g-1,n+2}\to\overline{\mathcal{M}}_{g,n}
\end{equation}
that glues the last two marked points. Standard formulas \cite[Ch. 17 Lemma 4.36]{ACG} in ${\rm{Pic}}_\qq\left({\overline{\mathcal{M}}}_{g-1,n+2}\right)$ give
\begin{equation}
\label{eq:canonical-pullback}
\begin{split}
\theta^*\delta_{\rm{irr}}&=\delta_{\rm{irr}}-\psi_{n+1}-\psi_{n+2}+\sum\delta_{i:\{n+1\}},\\
\theta^* K_{{\overline{\mathcal{M}}}_{g,n}}&=13\lambda+\psi_1+\ldots+\psi_n+2\psi_{n+1}+2\psi_{n+2}-2\delta-\delta_{1:\varnothing}-\delta_{0:\{n+1,n+2\}}
\end{split}
\end{equation}
Recall that if $X$ is a variety and $D$ a Cartier divisor on it, then the Iitaka dimension is defined as
\[ \kappa\left(X,D\right) = \dim \left( \bigoplus_{m=0}^{\infty} H^0(X,mD) \right) - 1 \]
when the quantity on the right is nonnegative, and is defined as $-\infty$ otherwise.

The following lemma is completely standard, see for example \cite{CR3, FV2}, but we include a proof for completeness.

\begin{lemma}
	\label{lemma1.sec3}
	Let $\gamma\in N_1(\overline{\mathcal{M}}_{g-1,n+2})$ be a nef curve class such that $(\gamma\cdot \theta^*\delta_{\rm{irr}})>0$. If 
	$$(\gamma\cdot \theta^*K_{\overline{\mathcal{M}}_{g,n}})<0$$ 
	then $\overline{\mathcal{M}}_{g,n}$ is uniruled. If instead $(\gamma\cdot \theta^*K_{\overline{\mathcal{M}}_{g,n}})=0$, then 
	$${\rm{Kod}}\left(\overline{\mathcal{M}}_{g,n}\right)\leq \kappa\left(\overline{\mathcal{M}}_{g-1,n+2}, \theta^* K_{\overline{\mathcal{M}}_{g,n}}\right).$$
\end{lemma}
\begin{proof}
	Suppose first that $(\gamma\cdot \theta^*K_{\overline{\mathcal{M}}_{g,n}})<0$ and assume that $mK_{\overline{\mathcal{M}}_{g,n}}$ is effective, for some $m>0$. Then this is numerically equivalent to a class $a\cdot \delta_{\rm{irr}}+D$, where $a\geq 0$ and $D$ is supported outside $\Delta_{\rm{irr}}$. Hence, $\theta^*D$ is an effective divisor on $\overline{\mathcal{M}}_{g-1,n+2}$ and then $(\gamma \cdot \theta^*D) \geq 0$. But then 
	\[ m(\gamma \cdot \theta^*K_{\overline{\mathcal{M}}_{g,n}}) = a\cdot (\gamma \cdot \theta^*\delta_{\rm{irr}}) + (\gamma \cdot \theta^*D) \geq 0 \]
	which is a contradiction.
	
	If instead $(\gamma\cdot \theta^*K_{\overline{\mathcal{M}}_{g,n}})=0$ the same reasoning shows that the class $mK_{\overline{\mathcal{M}}_{g,n}}-\delta_{\rm{irr}}$ is not effective  for any $m\geq 1$. From the exact sequence
	\begin{equation}
	\label{eq2.sec3}
	0\to H^0\left(\overline{\mathcal{M}}_{g,n},\mathcal{O}(mK_{\overline{\mathcal{M}}_{g,n}}-\delta_{\rm{irr}})\right)\to H^0\left(\overline{\mathcal{M}}_{g,n},\mathcal{O}(mK_{\overline{\mathcal{M}}_{g,n}})\right)\to H^0\left(\Delta_{\rm{irr}},\mathcal{O}(mK_{\overline{\mathcal{M}}_{g,n}})\right).
	\end{equation}
	we see that
	$$h^0\left(\overline{\mathcal{M}}_{g,n},\mathcal{O}(mK_{\overline{\mathcal{M}}_{g,n}})\right)\leq h^0\left(\Delta_{\rm{irr}},\mathcal{O}(mK_{\overline{\mathcal{M}}_{g,n}})\right).$$
	The map $\theta_{g,n}\colon \overline{\mathcal{M}}_{g-1,n+2}\to \Delta_{\rm{irr}}$ has degree two and it is simply ramified along $\Delta_{0:\{n+1,n+2\}}$. In any case, the following inequality holds:
	$$h^0\left(\Delta_{\rm{irr}},\mathcal{O}(mK_{\overline{\mathcal{M}}_{g,n}})\right)\leq h^0\left(\overline{\mathcal{M}}_{g-1,n+2}, \mathcal{O}(m\cdot \theta^\star K_{\overline{\mathcal{M}}_{g,n}})\right).$$
\end{proof}

Now we can conclude the proof of Theorem \ref{thm1}.

\begin{proposition}
	The moduli spaces  $\overline{\mathcal{M}}_{12,6}$ and $\overline{\mathcal{M}}_{12,7}$ are uniruled.
\end{proposition}
\begin{proof}
	Let $(S,H)$ be a general polarized K3 surface of genus eleven and $x_1,\ldots,x_{9}$ general points on $S$. Fix a general pencil $\Gamma \subseteq |\mathcal{O}_S(H)|$ passing through the points $x_i$. By Mukai's construction (see \cite{M} and the proof of Theorem \ref{thm:thm3again} below ) the pencil gives a nef curve class $\gamma\in\overline{\mathcal{M}}_{11,9}$ and the formulas \eqref{standard} give us that $\gamma$ intersects the generators of the Picard group of $\overline{\mathcal{M}}_{11,9}$ as follows 
	\begin{equation}
	\label{eq1.sec4}
	(\gamma\cdot\lambda)=12, \hspace{.3cm} (\gamma\cdot\psi_i)=1, \hspace{.3cm} (\gamma\cdot\delta_{\rm{irr}})= 84,\hspace{.3cm} \hbox{and}\hspace{.3cm} (\gamma\cdot\delta_{i:S})=0.
	\end{equation}
	Let us consider the  gluing map $\theta\colon \overline{\mathcal{M}}_{11,9}\longrightarrow\overline{\mathcal{M}}_{12,7}$: from the formulas in \eqref{eq:canonical-pullback} one computes $(\gamma\cdot\theta^*\delta_{\rm{irr}})=82$ and $(\gamma\cdot \theta^*K_{\overline{\mathcal{M}}_{12,7}})=-1$. Then, Lemma \ref{lemma1.sec3} shows that $\overline{\mathcal{M}}_{12,7}$ is uniruled and then $\overline{\mathcal{M}}_{12,6}$ is uniruled as well.
\end{proof}

Finally, we can bound the Kodaira dimension of $\overline{\mathcal{M}}_{12,8}$.

\begin{theorem}\label{thm:thm3again}
	The Kodaira dimension of $\overline{\mathcal{M}}_{12,8}$ is bounded by $\dim\left(\overline{\mathcal{M}}_{12,8}\right)-2$. 
\end{theorem}
\begin{proof}
	Let $(S,H)$ be a general polarized K3 surface of genus $11$ and $x_1,\dots,x_{10}$ general points. The pencil $\Gamma = |\mathcal{O}_S(C)\otimes \mathcal{I}_{x_1,\dots,x_{10}}|$ induces a nef curve class $\gamma$ on $\overline{\mathcal{M}}_{11,10}$ and for its pushforward along $\theta\colon \overline{\mathcal{M}}_{11,10} \to \overline{\mathcal{M}}_{12,8}$ we see that
	$$(\gamma\cdot \theta^*K_{\overline{\mathcal{M}}_{12,8}})=0\hspace{.5cm}\hbox{and}\hspace{.5cm}(\gamma\cdot\theta^*\delta_{\rm{irr}})=82.$$
	Then Lemma \ref{lemma1.sec3} shows that
	$${\rm{Kod}}\left(\overline{\mathcal{M}}_{12,8}\right)\leq \kappa\left(\overline{\mathcal{M}}_{11,10}, \theta^* K_{\overline{\mathcal{M}}_{12,8}}\right).$$
	We proceed to bound the right hand side. Recall Mukai's construction \cite{M}: there exists birationally a $\pp^1$-bundle $\mathcal{P}$ over the moduli space $\mathcal{F}_{11,10}$, where the general fiber over $(S,H,x_1,\ldots,x_{10})\in\mathcal{F}_{11,10}$ is the linear system of curves in $|H|$ passing through the points $x_i$
	\begin{equation}
	\label{eq2.sec4}
	\pp^1\cong \left|\mathcal{O}_S(H)\otimes\mathcal{I}_{x_1,\ldots,x_{10}}\right|.
	\end{equation}
	
	The natural map $\mathcal{P}\dashrightarrow \overline{\mathcal{M}}_{11,10}$ is birational and it is defined on the complete general fiber of $\mathcal{P}\to\mathcal{F}_{11,10}$, i.e., every $10$-pointed curve in the general linear system \eqref{eq2.sec4} is stable. Moreover, the push forward of the general fiber intersects $\theta^* K_{\overline{\mathcal{M}}_{12,8}}$ trivially. Thus,
	$$\kappa\left(\overline{\mathcal{M}}_{11,10}, \theta^* K_{\overline{\mathcal{M}}_{12,8}}\right)\leq \dim(\overline{\mathcal{M}}_{11,10})-1.$$ 
	
	and we conclude.
\end{proof}

\section{Bound on the Kodaira dimension of $\overline{\mathcal{M}}_{16}$}
\label{sec4}

Finally, we turn our attention to curves of genus $16$. The proof is the same as in \cite{FV2} but their key technical result \cite[Proposition 4]{FV2} is replaced by our Lemma \ref{lemma:genus15pencil}, which we proved via Proposition \ref{prop:smoothing}. Furthermore, we observe that the bound on the Kodaira dimension can be pushed below by one.

More precisely, we provide a bound by constructing nef curves that come from covering curves of $\overline{\mathcal{M}}_{15,2}$, pushed forward to $\overline{\mathcal{M}}_{16}$ via the $\Delta_{\rm{irr}}$-boundary map.

Thus, let $\left(C,p_1,p_2\right)$ be a general $2$-pointed curve of genus $15$ and 
$$\Gamma=\left|\mathcal{O}_{S}(C)\otimes\mathcal{I}_{p_1,p_2}\right|$$
the pencil of Lemma \ref{lemma:genus15pencil}. Standard computations and the formulas \eqref{standard} lead to the following intersection numbers between $\Gamma$ and the generators of $\Pic_\qq\left(\overline{\mathcal{M}}_{15,2}\right)$:
\begin{equation}\label{eq1.sec4}
(\Gamma\cdot\lambda)=22,\,\,\, (\Gamma\cdot \delta)=145, \,\,\, \hbox{and}\,\,\,(\Gamma\cdot \psi_1)=(\Gamma\cdot\psi_2)=1.
\end{equation}

\begin{theorem}
The Kodaira dimension of $\overline{\mathcal{M}}_{16}$ is bounded by $\dim \overline{\mathcal{M}}_{16}-2$.
\end{theorem}

\begin{proof}
Let $\theta:\overline{\mathcal{M}}_{15,2}\to \overline{\mathcal{M}}_{16}$ be the $\delta_{\rm{irr}}$-boundary map that glues the two marked points. By construction the curve $\Gamma$ is a covering curve for $\overline{\mathcal{M}}_{15,2}$. Now, from \eqref{eq:canonical-pullback} and \eqref{eq1.sec4} we have 
\begin{equation}
\label{eq:genus16intersection0}
(\Gamma\cdot \theta^*K_{\overline{\mathcal{M}}_{16}})  = 13(\Gamma\cdot\lambda)+2(\Gamma\cdot(\psi_1+\psi_2))-2(\Gamma\cdot\delta)-(\Gamma\cdot\delta_{1:\varnothing})-(\Gamma\cdot\delta_{0:\{1,2\}})=0 
\end{equation} 
and Lemma \ref{lemma1.sec3} applies; 
\begin{equation}
\label{eq2.sec4}
{\rm{Kod}}(\overline{\mathcal{M}}_{16})\leq \kappa\left(\overline{\mathcal{M}}_{15,2}, \theta^\star K_{\overline{\mathcal{M}}_{16}}\right).
\end{equation}
We will bound the right hand side.
Let $\mathcal{G}$ be the moduli space of smooth $(2,2,2,2)$ complete intersection surfaces $S\subseteq\pp^6$, together with two marked points $p_1,p_2\in S$, everything up to a linear change of coordinates. This space can be seen as the GIT quotient by $PGL_7$  
of an appropriate universal family over an open subset of ${\rm{Gr}}\left(4,H^0\left(\mathcal{O}_{\pp^6}(2)\right)\right)$.

Let also $\mathcal{W}$ be the space that parameterizes isomorphism classes of tuples $(C,L)$, where $C\in\mathcal{M}_{15}$ is a genus $15$ curve and $L\in W^1_9(C)$. This dominates $\mathcal{M}_{15}$ with one dimensional general fiber. We consider the incidence correspondence 
\begin{equation}
\label{eq3.sec4}
\begin{tikzcd}
&\Sigma\subset\mathcal{G}\times\mathcal{W}\arrow[dl, "p"']\arrow[dr, "q"]&\\
\mathcal{G}&&\mathcal{W}
\end{tikzcd}
\end{equation}
consisting of points $(S,p_1,p_2,C,L)\in \Sigma\subset\mathcal{G}\times \mathcal{W}$ such that $C\subset S$, $p_1,p_2\in C$, $\mathcal{O}_C(C)\cong L$ and $L$ does not separate the points $p_1,p_2$. We consider the fiber of the map $f\colon \Sigma\dashrightarrow \mathcal{M}_{15,2}$: for a general pointed curve $[C,p_1,p_2] \in \mathcal{M}_{15,2}$ the space $W^1_9(C)$ is one-dimensional and there are finitely many $L\in W^1_9(C)$ which do not separate $p_1,p_2$, see also \cite[Section 2, page 6]{FV2}. Then the discussion in Section \ref{sec.gen15}, shows that  each such $L$ determines uniquely an element $(S,p_1,p_2)\in \mathcal{G}$. Thus the map $f$ is generically finite. 
Moreover, the map $p$ is birationally a $\pp^1$-bundle over the image. The fiber over a general point in the image $p(S,C,p_1,p_2,L)=(S,p_1,p_2)\in\mathcal{G}$ is given by the pencil 
$$\Gamma \cong \left\{(S,C_t,p_1,p_2,L)\mid C_t\in|\mathcal{O}_S(C)\otimes\mathcal{I}_{p_1,p_2}|\right\},$$
where $C_t$ is always stable, cf. Lemma \ref{lemma:genus15pencil}. Let $B\subset\mathcal{G}$ be the image of $p$ and consider the generically finite map to $\overline{\mathcal{M}}_{15,2}$
\begin{equation}
\label{eq4.sec4}
\begin{tikzcd}
\Sigma\arrow[dashed, r, "f"]\arrow[d, "p"]&\overline{\mathcal{M}}_{15,2}\\
B.&
\end{tikzcd}
\end{equation}
The map $f$ is defined over the complete general fiber $F$ of $p$, and the computations of (\ref{eq:genus16intersection0}) show that
$$(f_* F\cdot \theta^* K_{\overline{\mathcal{M}}_{16}})=0.$$
Thus, the Iitaka dimension of $\theta^* K_{\overline{\mathcal{M}}_{16}}$ is bounded by the dimension of $B$ and this together with the inequality \eqref{eq2.sec4} give us the desired bound.
\end{proof}

\section*{Acknowledgements}
 We would like to thank Gabi Farkas and Sandro Verra for many interesting conversations, correspondence, and for sharing with us the results of \cite{FV2}. We are grateful to the anonymous referee for their useful comments and remarks. The second author would also like to thank Pedro Montero and Jenia Tevelev for helpful conversations.



\begin{thebibliography}{aaaaa}

\bibitem[AB]{MacaulayCode} D. Agostini, I. Barros, auxiliary Macaulay2 code available at \url{https://personal-homepages.mis.mpg.de/agostini/} (2020). 

\bibitem[ACG]{ACG} E. Arbarello, M. Cornalba, and P. A. Griffiths, {\em{Geometry of algebraic curves. Vol. II}}, Grundlehren der Math. Wiss., vol. 267, Springer-Verlag, New York, 2011.

\bibitem[BM]{BM} I. Barros and S. Mullane, {\em{Two moduli spaces of Calabi-Yau type}}, Int. Math. Res. Not. IMRN (2019), rnz264, https://doi.org/10.1093/imrn/rnz264.

\bibitem[BL]{BL} M. Beltrametti and A. Lanteri, {\em  On the 2 and the 3-connectedness of ample divisors on a surface}, Manuscripta math. {\bf{58}} (1987), 109--128.

\bibitem[Be]{Be} L. Benzo, {\em{Uniruledness of some moduli spaces of stable pointed curves}}, J. Pure Appl. Algebra {\bf{218}} no. 3, (2014), 395--404.

\bibitem[BDPP]{BDPP} S. Boucksom, J.-P. Demailly, M. P\u{a}un, T. Peternell, {\em{The pseudoeffective cone of a compact K\"{a}hler manifold and varieties of negative Kodaira dimension}}, J. Algebraic Geom. {\bf{22}} (2013), 201--248.

\bibitem[BV]{BV} A. Bruno, A. Verra, {\em{$\mathcal{M}_{15}$ is rationally connected}}, in {\em{Projective varieties with unexpected properties}}, vol. 51--65, Walter de Gruyter GmbH, Berlin, 2005.


\bibitem[Ca]{Ca} G. Casnati, {\em{On the rationality of moduli spaces of pointed hyperelliptic curves}},
Rocky Mountain J. Math. {\bf{42}} no. 2, (2012), 491--498.

\bibitem[CR1]{CR1} M. C. Chang and Z. Ran, {\em{Unirationality of the moduli space of curves of genus $11, 13$ (and $12$)}},   Invent. math. {\textbf{76}} (1984), 41--54.
\bibitem[CR2]{CR2} M. C. Chang and Z. Ran, {\em{The Kodaira dimension of the moduli space of curves of genus $15$}}, J. Differential Geom. {\textbf{24}} (1986), 205--220.
\bibitem[CR3]{CR3} M. C. Chang and Z. Ran, {\em{On the slope and Kodaira dimension of $\overline{\mathcal{M}}_g$ for small $g$}}, J. Differential Geom. {\textbf{34}} (1991), 267--274.

\bibitem[dR]{dR} S. di Rocco, {\em k-Very ample line bundles on del Pezzo surfaces}, Math. Nachr. {\bf{179}}. 47--56.

\bibitem[EH1]{EH1} D. Eisenbud and J. Harris, {\em{Limit linear series: Basic theory}}, Invent. math. {\textbf{85}} (1986), 337--371.
\bibitem[EH2]{EH2} D. Eisenbud and J. Harris, {\em{The Kodaira dimension of the moduli space of curves of genus $\geq23$}},  Invent. math. {\textbf{90}} (1987), 359--387.

\bibitem[F]{F} G. Farkas, {\em{Koszul divisors on moduli spaces of curves}}, American Journal of Mathematics {\textbf{131}} (2009), 819--867.

\bibitem[FJP]{FJP} G. Farkas, D. Jensen, S. Payne, {\em{The Kodaira dimensions of $\overline{\mathcal{M}}_{22}$ and $\overline{\mathcal{M}}_{23}$}}, preprint (2020), ArXiv: 2005.00622.

\bibitem[FPo]{FPo} G. Farkas and M. Popa, {\em{Effective divisors on $\overline{\mathcal{M}}_{g}$, curves on $K3$ surfaces and the Slope Conjecture}}, Journal of Algebraic Geometry {\textbf{14}} (2005), 151--174. 

\bibitem[FV1]{FV1} G. Farkas and A. Verra, {\em{The classification of universal Jacobians over the moduli space of curves}}, Commentarii Mathematici Helvetici {\bf{88}} (2013), 587--611.

\bibitem[FV2]{FV2} G. Farkas and A. Verra, {\em{On the Kodaira dimension of $\overline{\mathcal{M}}_{16}$}}, preprint (2020), ArXiv: 2008.08852.

\bibitem[Fr]{Fr} R. Friedman, {\em{Global smoothings of varieties with normal crossings}}, Ann. of Math. (2) {\bf{118}} (1983), no. 1, 75--114.

\bibitem[M2]{M2} D. Grayson and M. Stillman, {\em{Macaulay2, a software system for research in algebraic geometry}}, available at \texttt{http://www.math.uiuc.edu/Macaulay2/}.

\bibitem[Ha]{Ha} P. Hacking, {\em{Compact moduli of surfaces of general type}}, Contemp. Math. {\bf{564}} (2012), 1--18.

\bibitem[HM]{HM} J. Harris and D. Mumford, {\em{On the Kodaira dimension of $\overline{\mathcal{M}}_{g}$}}, Invent. math. \textbf{67} (1982), 23--88.

\bibitem[KT]{KT} H. Keneshlou and F. Tanturri, {\em{On the unirationality of moduli spaces of pointed curves}}, preprint (2020), ArXiv: 2003.07888.

\bibitem[Li]{Li} J. Lipman, {\em{Notes on derived functors and Grothendieck duality}}, in: Foundationsof Grothendieck Duality for Diagrams of Schemes, Lecture Notes in Mathematics Vol. 1960 (2009), 1--259.  

\bibitem[Lo]{Lo} A. Logan, {\em{The Kodaira dimension of moduli spaces of curves with marked points}}, American Journal of Mathematics {\bf{125}} (2003), no. 1, 105--138. 

\bibitem[M]{M} S. Mukai, {\em{Curves and K3 surfaces of genus eleven}}, in: Moduli of vector bundles, Lecture Notes in Pure and Applied Mathematics Vol. 179 (1996), 189--197.

\bibitem[Sc]{Sc} I. Schwarz, {\em{On the Kodaira dimension of the moduli space of hyperelliptic curves with marked points}}, preprint (2020), ArXiv: 2002.03417.

\bibitem[Se]{Se} E. Sernesi, {\em{Deformations of algebraic schemes}}, Grundlehren der Mathematischen Wissenschaften Vol. 334, Springer-Verlag 2006.

\bibitem[S-P]{S-P} The Stacks project authors, {\em{The Stacks project}}, 2020, https://stacks.math.columbia.edu.

\bibitem[T]{T} D. Tseng, {\em{On the slope of the moduli space of genus $15$ and $16$ curves}}, preprint (2019), ArXiv: 1905.00449.

\bibitem[V]{V} A. Verra, {\em{The unirationality of the moduli space of curves of genus $14$ and lower}},
Compositio Math. {\bf{141}} (2005), 1425--1444.

\end{thebibliography}
\end{document}